\documentclass[11pt]{amsart}
\usepackage{amssymb}
\usepackage{graphicx}
\usepackage{xcolor} % A package to add color.
\usepackage{tensor}
\usepackage{fullpage} % Sets all margins to 1 in.  
\usepackage{amsmath}
\usepackage{amsthm}
\usepackage{verbatim}
\usepackage{txfonts, bm, mathtools}
\usepackage{dsfont}
\usepackage{mathrsfs}
\usepackage{appendix}

\usepackage{enumitem}
\setlist[enumerate]{leftmargin=1.5em}
\setlist[itemize]{leftmargin=1.5em}

\usepackage{hyperref}
\usepackage{seqsplit,mathtools} 
\usepackage{caption}
\usepackage{subcaption}
\usepackage{ulem}

\usepackage{thmtools}
\usepackage{thm-restate}
\usepackage{lipsum}

\usepackage{dirtytalk}
\usepackage{tikz}
\newcommand*\circled[1]{\tikz[baseline=(char.base)]{
        \node[shape=circle,draw,inner sep=1pt] (char) {#1};}}

\mathtoolsset{showonlyrefs}

\definecolor{green}{rgb}{0,0.5,0} % Redefines the color green.
\newcommand{\dd}{\textrm{d}}

\newtheorem{thm}{Theorem}[section]

\newtheorem{lem}[thm]{Lemma}
\newtheorem{prop}[thm]{Proposition}

\newtheorem{defn}[thm]{Definition}

\newtheorem{rmk}[thm]{Remark}

\theoremstyle{definition}

\numberwithin{equation}{section}
\newcommand{\nrm}[1]{\Vert#1\Vert}

\newcommand{\set}[1]{\{#1\}}
\newcommand{\tld}[1]{\widetilde{#1}}

\newcommand{\nnrm}[1]{{\vert\kern-0.25ex\vert\kern-0.25ex\vert #1 
\vert\kern-0.25ex\vert\kern-0.25ex\vert}}

\newcommand{\supp}{{\mathrm{supp}}\,}

\newcommand{\lap}{\Delta}

\newcommand{\rd}{\partial}
\newcommand{\nb}{\nabla}

\newcommand{\imp}{\Rightarrow}

\newcommand{\ift}{\infty}

\newcommand{\alp}{\alpha}

\newcommand{\gmm}{\gamma}
\newcommand{\Gmm}{\Gamma}
\newcommand{\dlt}{\delta}

\newcommand{\eps}{\varepsilon}

\newcommand{\lmb}{\lambda}

\newcommand{\sgm}{\sigma}

\newcommand{\tht}{\theta}

\newcommand{\vtht}{\vartheta}
\newcommand{\omg}{\omega}
\newcommand{\Omg}{\Omega}

\newcommand{\bfe}{{\bf e}}

\newcommand{\bfu}{{\bf u}}
\newcommand{\bfv}{{\bf v}}

\newcommand{\bfx}{{\bf x}}
\newcommand{\bfy}{{\bf y}}
\newcommand{\bfz}{{\bf z}}

\newcommand{\bbR}{\mathbb R}

\newcommand{\calA}{\mathcal A}

\newcommand{\calG}{\mathcal G}

\newcommand{\calK}{\mathcal K}
\newcommand{\calL}{\mathcal L}

\newcommand{\calS}{\mathcal S}

\newcommand{\bfomg}{\boldsymbol{\omg}}		% fluid vorticity for the nonlinear Hall-MHD
\newcommand{\f}[2]{\frac{#1}{#2}}       %fraction
\newcommand{\ii}[2]{\int_{#1}^{#2}}     %integral

\newcommand{\q}{\mbox{ }}       %empty spaces
\newcommand{\qd}{\quad }

\newcommand{\bfone}{\mathbf{1}}

\begin{document}

\title{Vortex atmospheres of traveling vortices:\\ rigorous definition, existence, and topological classification}

\author{Kyudong Choi}
\address{Department of Mathematical Sciences, Ulsan National Institute of Science and Technology, 50 UNIST-gil, Eonyang-eup, Ulju-gun, Ulsan 44919, Republic of Korea.}
\email{kchoi@unist.ac.kr}

\author{In-Jee Jeong}
\address{Department of Mathematical Sciences and RIM, Seoul National University, 1 Gwanak-ro, Gwanak-gu, Seoul 08826, and School of Mathematics, Korea Institute for Advanced Study, Republic of Korea.}
\email{injee$ \_ $j@snu.ac.kr}

\author{Young-Jin Sim}
\address{Department of Mathematical Sciences, Ulsan National Institute of Science and Technology, 50 UNIST-gil, Eonyang-eup, Ulju-gun, Ulsan 44919, Republic of Korea.}
\email{yj.sim@unist.ac.kr}

%\subjclass[2020]{35Q31, 35Q35}

\date\today

\begin{abstract}
In incompressible and inviscid fluids, the vortex atmosphere refers to the collection of fluid particles outside the support of a traveling vortex that are nevertheless carried along
with it. This phenomenon has been recognized since the nineteenth century, e.g., in the classical works of O. Reynolds [Nature, 1876] and O. Lodge [Lond. Edinb. Dubl. Phil. Mag., 1885], yet rigorous mathematical definitions and proofs have remained largely undeveloped, with most subsequent studies relying on thin-core approximations or asymptotic analyses.
In this paper, we give a rigorous definition of a vortex atmosphere and establish its existence and uniqueness. We further compare the planar atmosphere surrounding a 2D vortex dipole with the axisymmetric atmosphere surrounding a 3D vortex ring. In particular, we emphasize and prove the topological distinctions
observed by W. Hicks [Lond. Edinb. Dubl. Phil. Mag., 1919]: under natural assumptions, every  2D dipole with its atmosphere forms an
oval-shaped region, whereas for 3D rings, both spheroidal and toroidal configurations may occur.
 Our proof is based on showing that each atmosphere can be characterized precisely as a specific superlevel set of its corresponding stream function.
\end{abstract}
%\keyword{dsdsd}
\maketitle

\section{Introduction}

A \textit{traveling vortex} is an Euler flow whose vorticity distribution translates rigidly at a constant velocity. The region where the vorticity is nonzero—called the vortex core—maintains its shape while being transported. Typical examples include a three-dimensional axisymmetric vortex ring (e.g., a smoke ring in air or a bubble ring in water) and a counter-rotating symmetric vortex pair (a dipole) in two dimensions. Although classical studies have revealed many properties of the vortex core, they offer little insight into \textit{how the surrounding irrotational fluid is carried along}. In this paper, we revisit this aspect with a precise analysis of the transport structure outside the core.

\medskip

A key feature of this phenomenon is that the vortex induces a coherent transport of irrotational fluid outside its core. As a vortex moves through a fluid, it may entrain part of the irrotational flow around it. In particular, some fluid particles outside the core revolve around it and are carried along with it; these particles constitute what may be informally viewed as a free-rider region, since they are transported by the vortex without contributing to its momentum. Others, however, fail to match the vortex’s speed and inevitably lag behind.

\medskip
From this perspective, we may distinguish three regions of fluid with distinct characteristics\footnote{for instance, suggested by Hicks \cite{Hicks1919}.}:

\begin{enumerate}
\item 
The rotational fluid that constitutes the vortex core.
\medskip
\item \textit{(Free-rider region)} 
The irrotational fluid outside the core whose particles are transported forward together with region (1).
\medskip
\item 
The irrotational fluid that fails to travel with the core and eventually lags behind.
\end{enumerate}

\medskip\noindent 
From a mathematical viewpoint, the region (2) introduces a new type of invariant region for the Euler flow—one defined not by vorticity but by dynamical transport.
To the best of our knowledge, this notion has not been formalized in a fully rigorous way.
Nevertheless, the underlying phenomenon has long been recognized in hydrodynamics.
In particular, it was already implicitly observed by Reynolds \cite{Reynolds1876} in 1876, in describing the entrainment and carriage of ambient fluid by a vortex ring, although no mathematical formulation was given. He said
\textit{\say{$\cdots$they are continually adding to their bulk water taken up from that which 
surrounds them and with which their forward momentum has to be shared.}} and \textit{\say{The form of the mass of fluid moving forward is not nearly that of the ring, but is an oblate spheroid a good deal longer than this ring which it encloses.}}\footnote{quoted from \cite[pg. 478]{Reynolds1876}.}
Around the same time, Lodge \cite{Lodge1885} in 1885 investigated the relation between the vortex ring's speed and the surrounding streamline patterns of the regions (2) and (3). This region (2) has been referred to in the literature as an \say{atmosphere} (Eisenga \cite{Eisenga1997}, Akhmetov \cite{Akhmetov2009}, Meleshko et al. \cite{MKGK1992, MGK2011}), a \say{vortex bubble} (Sullivan et al. \cite{SNHBD2008}), or a \say{vortex cloud} (Meleshko et al. \cite{MGK2011}); see \cite{MGK2011} for a comprehensive literature review. 

\medskip 
In 1919, Hicks began his paper \cite{Hicks1919} by noting that \say{\textit{the distinction between the rotational region (1) and the irrotational regions (2) and (3) is fundamental and well known. Less attention than it deserves, however, appears to have been devoted to the discussion of the relationship between (2) and (3).}} His remark continues to hold from the viewpoint of rigorous mathematical analysis. As Batchelor in 1967 emphasized in \cite[pg.~523]{Batchelor1967}, \say{\textit{mathematical analysis of such vortex rings is made difficult by ignorance of the exact shape of the cross-sections of the tube containing the vorticity that is compatible with steady motion.}} Consequently, most existing studies on the dynamics of regions (2) and (3) rely on thin-core approximations or numerical computations to infer streamline configurations of vortex rings, and formal mathematical definitions or proofs have been largely absent.

\medskip
Hicks \cite{Hicks1919} conducted detailed investigations and provided a classification of various geometric forms of the region (1)+(2) for a vortex ring with small cross-section, comparing the results with the 2D analogy case: a counter-rotating vortex pair (a {vortex dipole}). He observed that the region (1)+(2) of a 2D point vortex pair forms a single simply-connected component of oval shape surrounding the pair. By contrast, in the case of a vortex ring whose cross-sectional radius is sufficiently small and gradually increases, he found that the fluid speed at the ring’s center on the axis of symmetry—initially smaller than the ring’s traveling speed—eventually increases beyond it.
Depending on whether the speed at the ring's center is less than, equal to, or greater than the ring’s traveling speed, Hicks observed that the region (1)+(2) admits three distinct configurations, namely, a toroidal ring (Figure~\ref{subfig: toroid}), a revolved lemniscate\footnote{A curve having a $\infty$–shape.} (Figure~\ref{subfig: lemniscate}), and a spheroid (Figure~\ref{subfig: spheroid}), in this order. Also see the recent work of Masroor \& Stremler \cite{MS2022}.

\medskip
Subsequent studies reported similar structural transitions in vortex rings. For instance, we refer to Fig. 7.2.4 in pg. 525 in the classical exposition  
Batchelor \cite{Batchelor1967} (see also the more recent review by Wu et al. \cite[Fig. 7.3, pg. 218]{WMZ2015}).
Among them, we emphasize that Norbury  \cite{Norbury1973} in 1973 constructed and numerically described a family of steady vortex rings in which region (1) has a toroidal shape (\cite[Figure 3]{Norbury1973}) and region (1)+(2) takes a spheroidal form (\cite[Figure 4]{Norbury1973}).  For experimental studies, we refer to the classical work of Maxworthy \cite{Maxworthy1972} in 1972 (see also Dabiri \& Gharib \cite{Dabiri2008} and references therein), which demonstrates that vortex rings entrain a considerable amount of surrounding fluid as they propagate.

\begin{figure}
    \centering
\begin{subfigure}[t]{0.32\textwidth}
\includegraphics[width=\linewidth, height=3.5cm]{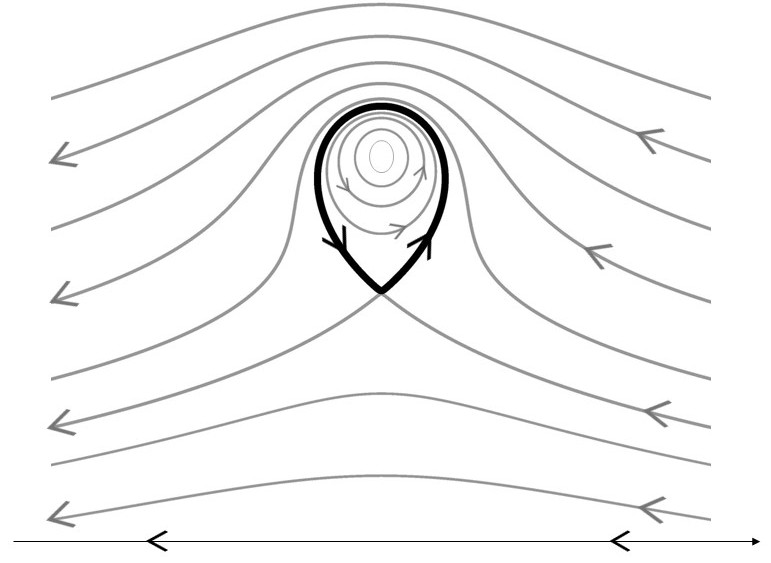}
\caption{}
\label{subfig: toroid}
\end{subfigure}
\begin{subfigure}[t]{0.32\textwidth}
\includegraphics[width=\linewidth, height=3cm]{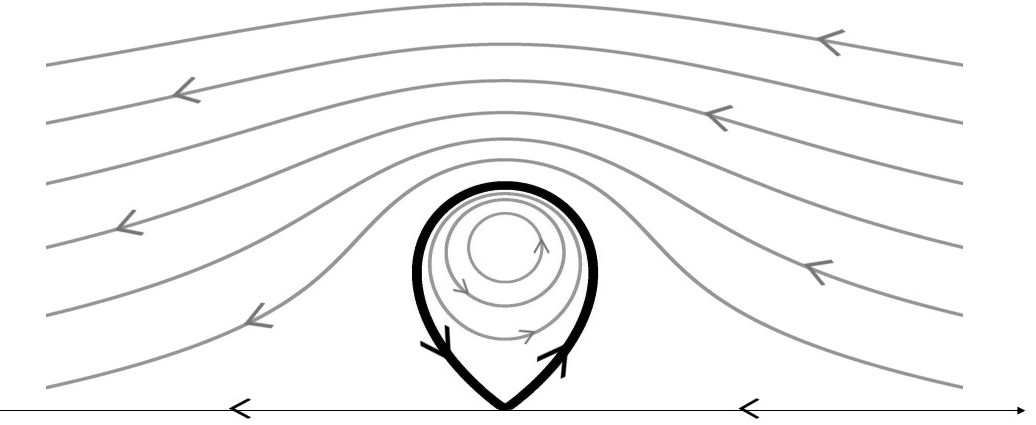}
\caption{}
\label{subfig: lemniscate}
\end{subfigure}
\begin{subfigure}[t]{0.34\textwidth}
\includegraphics[width=\linewidth, height=2.1cm]{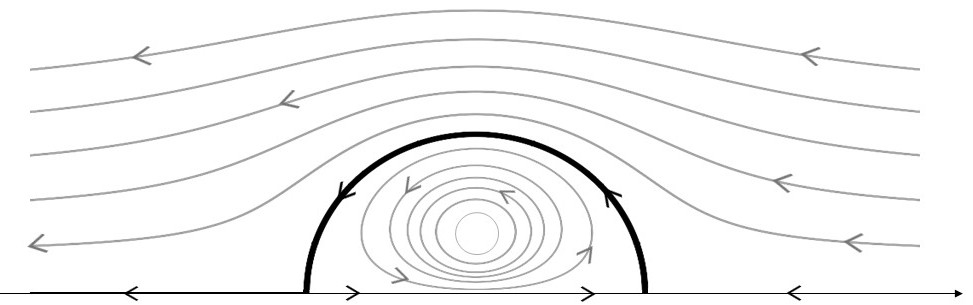}
\caption{}
\label{subfig: spheroid}
\end{subfigure}
\caption{The above figures describe the cross-sections of the region (1)+(2) observed by 
Hicks. In each figure, the rotation axis lies horizontally at the bottom. The region (1)+(2) of a vortex ring can take the form of (A) a toroidal ring, (B) a revoloved lemniscate, or (C) a spheroid.}
\label{fig: Hicks}
\end{figure}

\medskip 
The present paper provides a mathematical treatment of Hicks’s work \cite{Hicks1919}. More precisely, we give rigorous definitions of a steady vortex ring, its two-dimensional analogue (a vortex dipole), and the associated regions (1)+(2) and (2) in each case. Under natural assumptions, we then investigate the geometric configurations and properties of region (1)+(2), highlighting fundamental differences in the flow patterns between rings and dipoles. This yields a rigorous confirmation of several observations made in \cite{Hicks1919} and in subsequent studies.

\medskip 
From now on, we refer 
to region (1) as the \textbf{{vortex core}},
to region (2) as the \textbf{{vortex atmosphere}}, and the union of regions (1) and   (2) as the \textbf{{vortex domain}}\footnote{To the best of our knowledge, there is no standard term in the literature referring to the union of regions (1) and (2); we therefore introduce the term “vortex domain” here.}. In what follows, we outline the structure of the paper and summarize the main results.

\medskip
\begin{itemize}
\item \textbf{Section~\ref{sec: math setting}}: Mathematical framework.
\begin{itemize}
\item Definition of dipole and ring (and their core): 
Definition~\ref{def: 3D ring}, \ref{def: 2d dipole}
    \item Definition of atmosphere (and domain): Definition~\ref{def: 3D atmosphere}, \ref{def: 2D atmosphere}
    \item Existence: Proposition~\ref{prop: existence}
    \item Uniqueness: Remark~\ref{rmk: unique}
\end{itemize}
\medskip
\item \textbf{Section~\ref{sec: compare ring and dipole}}: Comparison of representative examples of vortex domains.

\textit{We analyze the well-known examples, which are symmetrically concentrated along the $x_2$-axis for dipoles ($r$-axis for rings).}

\begin{itemize}
 \item 2D Vortex dipole: oval domain (Theorem~\ref{thm: 2D atmosphere})
    \item 3D Vortex ring: 
    spheroidal domain (Theorem~\ref{thm: 3D sphere}),  toroidal domain (Theorem~\ref{thm: 3D donut})
   \end{itemize}
\medskip
\item \textbf{Section~\ref{subsec: classif. of atmos.}}: Classification of vortex domains and their superlevel set characterizations.

\textit{We classify vortex domains whose core is simply connected in the half-plane $\{x_2>0\}$ for dipoles ($\{r>0\}$ for rings), and is symmetrically concentrated along the $x_2$-axis  ($r$-axis).}

\begin{itemize}
    \item 2D Vortex dipole: oval domain  (Theorem~\ref{thm: class. of 2D atmos.})
     \item 3D Vortex ring: spheroidal domain, toroidal domain, and revolved-lemniscate domain (Theorem~\ref{thm: class. of 3D atmos.})
\end{itemize}
\medskip
\noindent$\triangleright$ \textit{Extra in Section~\ref{subsec: classif. of atmos.}}: A special 2D vortex dipole.

The vortex atmosphere of a 2D dipole is empty if and only if it is a Sadovskii vortex (a touching pair of counter-rotating vortices; see \cite{CJS2025, CSW2025})
\medskip
\item\textbf{Section~\ref{sec: open problems}}: Open problems related to a vortex domain and atmosphere.
\end{itemize}

\medskip
\subsection{Underlying mechanisms and main assumptions}\q

\medskip

Our primary approach is to analyze the streamline configuration in the moving frame. By viewing a traveling vortex as a stationary object in this frame, we describe the corresponding vortex domain as the set of all fluid particles whose trajectories form bounded streamlines, thereby distinguishing those that are genuinely captured and transported by the vortex from those that merely pass by. This characterization allows us to identify each vortex domain as a specific superlevel set of the relative stream function.

\medskip
Within this framework, we also describe the possible geometric shapes of a given vortex domain. As observed in previous studies, whether the fluid speed at the ring’s center is smaller, equal to, or greater than the ring's traveling speed determines whether particles near the center remain with the ring or lag behind, and this distinction dictates the resulting configuration of the vortex domain. Our analysis shows that this dependence on the center speed arises from the monotonicity and directional bias of the velocity field in the irrotational region. In this region, the stream function satisfies a homogeneous second-order linear equation, which allows us to determine the signs of certain first and second derivatives. These sign properties yield monotonicity and directional information for the velocity components. Comparing the center speed with the traveling speed then reveals precisely how particles near the center move, thereby enabling a rigorous determination of the geometric shape of the vortex domain.

\medskip
The relationship between the center speed and the traveling speed--noted as the determinant of the vortex domain geometry--is precisely what distinguishes a two-dimensional vortex \textbf{dipole} from a three-dimensional vortex ring. For a 2D dipole, the center speed strictly exceeds twice its translation speed. Consequently, its vortex domain must contain the center and is therefore always an oval-shaped region surrounding the dipole. This phenomenon arises from the structure of the dipole: each of its two vortices induces a flow field that reinforces the motion of its counterpart, producing the strongest superposition at the dipole’s midpoint; see Figure~\ref{fig: dipole center}.
\medskip

In contrast, for a 3D \textbf{ring}, although each cross-section in the meridional half-plane resembles the upper piece of a counter-rotating dipole, the ring constitutes a connected vorticity distribution in 3D rather than two localized vortices. Hence, the ring’s center is not directly influenced by the surrounding flow in the same manner. As a result, the possible configurations of vortex domains vary in the three-dimensional ring case, as illustrated in Figure~\ref{fig: Hicks}, unlike the rigidly oval structure arising in the two-dimensional dipole in Figure~\ref{fig: dipole center}.

\medskip
\begin{figure}
\centering
\includegraphics[width=0.85\linewidth]{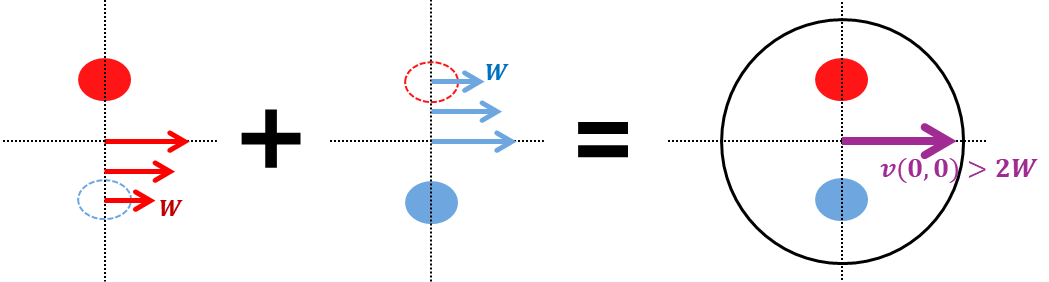}
\caption{
The figure illustrates a two-dimensional vortex dipole composed of two (odd) symmetric vortices, shown in red and blue. Their mutual induction generates a constant translation speed 
$W>0$. The induced velocities reinforce each other most strongly at the dipole’s midpoint, causing the center speed to exceed 
$2W$. Consequently, fluid particles located at the dipole’s center are advected along with the dipole, producing an oval-shaped vortex domain whose boundary is indicated by the solid black curve.
}
\label{fig: dipole center}
\end{figure}

In Section~\ref{sec: compare ring and dipole}, we impose Steiner symmetry on the vorticity profiles. For instance, the vorticity of a vortex ring is required to be symmetrically concentrated about a plane orthogonal to the symmetry axis; see Definition~\ref{def: 3D steiner}. This assumption is natural from a hydrodynamical viewpoint: physically stable steady vortices typically arise as maximizers of kinetic energy under suitable constraints, and Steiner symmetry is known to be a necessary condition for such maximization. A more detailed discussion of this variational interpretation is provided in Remark~\ref{rmk: Steiner is natural}.

\medskip

Section~\ref{subsec: classif. of atmos.} develops this analysis under the additional assumption that the vortex core is simply connected. More precisely, we assume that the cross-section of the core--taken in the meridional half-plane for a 3D ring or in the half-plane for a 2D dipole--is simply connected. To the best of our knowledge, no traveling dipole or ring with a multiply connected core has ever been found. Moreover, most known traveling vortex solutions possess simply-connected cores; examples include Hill’s spherical vortex \cite{Hill1894},   Chaplygin–Lamb dipole \cite{Chap1903, Lamb1993}, Norbury’s steady vortex rings \cite{Norbury1973}, the Sadovskii-type vortex patch constructed by Huang and Tong \cite{HT2025}, the vortex rings of Friedman and Turkington \cite{FT1981} and of Cao et al. \cite{CQYZZ2023}, and the vortex dipoles of Turkington \cite{Turkington1983}, Hmidi \& Mateu \cite{HM2017} (see also \cite{HH2021}), Cao et al. \cite{CLZ2021} (see Remark~2.14 in that paper), and Hassainia \& Wheeler \cite{HW2022}.

\medskip
We start our discussion in Section~\ref{sec: math setting} below by introducing a mathematical background that will be used throughout this paper.

\section{Mathematical setting}\label{sec: math setting}

In this paper, we consider incompressible and inviscid fluids described by the Euler equations of the vorticity form in $\bbR^3$: 
\begin{equation}\label{eq: 3D Euler}
\left\{
\begin{aligned}
&\q\rd_t\bfomg+(\bfv\cdot\nb)\bfomg=(\bfomg\cdot\nb)\bfv\qd\mbox{in}\q[0,\ift)\times\bbR^3,\\
&\q\nb\cdot\bfv=0,\q\nb\times\bfv=\bfomg,\\
&\q\bfomg|_{t=0}=\bfomg_0.
\end{aligned}\right.
\end{equation} Here, $\bfomg:[0,\ift)\times\bbR^3\to\bbR^3$ is the 3D vorticity and $\bfv:[0,\ift)\times\bbR^3\to\bbR^3$ is the 3D velocity that is recovered from its vorticity $\bfomg$ by the 3D Biot-Savart law $\bfv=\nb\times(-\lap_{\bbR^3})^{-1}\bfomg$.

\subsection{Vortex dipole and vortex ring}\q

\medskip
One of the well-known solutions of the 3D Euler equations \eqref{eq: 3D Euler} is an axisymmetric vortex ring without swirl, which steadily propagates along the symmetry axis at a constant speed. Its 2D analogue is a counter-rotating vortex dipole, consisting of a pair of odd-symmetric vortices that are likewise translated at a steady speed in the direction of the symmetry axis. In the sequel, we provide the background and preliminaries needed for a rigorous treatment of steady dipoles and rings.

\subsubsection{Dipole}\q

\medskip
In $\bbR^2$, the equations $\eqref{eq: 3D Euler}$ are  modified as 
\begin{equation}\label{eq: 2D Euler}
\left\{
\begin{aligned}
&\q\rd_t\omg+\bfu\cdot\nb\omg=0\qd\mbox{in}\q[0,\ift)\times\bbR^2,\\
&\q\bfu=\calK_2[\omg],\\
&\q\omg|_{t=0}=\omg_0
\end{aligned}
\right.
\end{equation} where $\omg:[0,\ift)\times\bbR^2\to\bbR$ is the 2D vorticity and $\bfu:[0,\ift)\times\bbR^2\to\bbR^2$ is the 2D velocity given by the 2D Biot-Savart law $\bfu=\calK_2[\omg]=\nb^\perp(-\lap_{\bbR^2})^{-1}\omg$ where $\nb^\perp=(\rd_{x_2},-\rd_{x_1})$. 
To model a counter-rotating dipole, we consider a traveling vortex solution to \eqref{eq: 2D Euler} having the odd-symmetry.

\begin{defn}[Vortex dipole]\label{def: 2d dipole}
A function $\overline{\omg}\in L^\ift(\bbR^2;\bbR)$ is called a vortex dipole if it is odd-symmetric: 
$$\overline\omg(x_1,x_2)=-\overline\omg(x_1,-x_2)\qd\mbox{for each}\q(x_1,x_2)\in\bbR^2,$$ non-negative in the half plane $\set{\bfx\in\bbR^2:x_2>0}$, and satisfies the following properties:\\

\noindent(i) The set $\set{\bfx\in\bbR^2:\overline\omg(\bfx)\neq0}$ is a bounded open subset of $\bbR^2$ (up to a set of measure zero). It is called the core of the vortex dipole $\overline\omg$.\\

\noindent(ii) It is a traveling solution to \eqref{eq: 2D Euler} in the sense that, for some constant $W>0$, the function $\omg:[0,\ift)\times\bbR^2\to\bbR$ defined by $\omg(t,
(x_1,x_2))=\overline\omg(x_1-Wt,x_2)$ solves \eqref{eq: 2D Euler}. The constant $W$ is called the traveling speed of the vortex dipole $\overline\omg$.
\end{defn}

\subsubsection{Ring}\q

\medskip
For an axisymmetric fluid without swirl, the 3D vorticity $\bfomg$ in \eqref{eq: 3D Euler} admits its angular component $\omg^\tht$ only in the usual cylindrical coordinate system $\bfx=(r,\tht,z)$, i.e., 
$$\bfomg=\omg^\tht(z,r)\cdot(-\sin\tht,\cos\tht,0).$$ Then we can rewrite \eqref{eq: 3D Euler} in terms of the \textit{relative vorticity} $\xi$, defined as  $\xi(z,r):=\omg^\tht(z,r)/r$, to obtain
\begin{equation}\label{eq: 3D axi. Euler}
\left\{
\begin{aligned}
&\q\rd_t\xi+\bfv\cdot\nb\xi\q=\q0\qd\mbox{in}\q[0,\ift)\times\bbR^3,\\
&\q\bfv=\calK_3[\xi],\\
&\q\q\xi|_{t=0}\q=\q \xi_0,
\end{aligned}
\right.
\end{equation} where $\bfv=\calK_3[\xi]$ denotes the 3D axi-symmetric Biot-Savart law.
\begin{defn}(Vortex ring)\label{def: 3D ring}
A non-negative function $\overline{\xi}\in L^\ift(\bbR^3;\bbR)$ is called a vortex ring if it is axisymmetric and satisfies the following properties:\\

\noindent(i) The set $\set{\bfx\in\bbR^3:\overline\xi(\bfx)>0}$ is a bounded open subset of $\bbR^3$ (up to a set of measure zero). It is called the core of the vortex ring $\overline{\xi}$.\\

\noindent(ii) It is a traveling solution to \eqref{eq: 3D axi. Euler} in the sense that, for some constant $W>0$, the axisymmetric function $\xi:[0,\ift)\times\bbR^3\to\bbR$ defined by $\xi(t,(z,r))=\overline\xi(z-Wt,r)$ solves \eqref{eq: 3D axi. Euler}. The constant $W$ is called the traveling speed of the vortex ring $\overline\xi$.
\end{defn}

\subsection{Vortex domain and atmosphere}\q

\medskip
We give a rigorous definition of the vortex atmosphere, a notion that, to our knowledge, has not been formalized before.
\begin{defn}\label{def: 3D atmosphere}
For any vortex ring $\overline\xi$ (Definition~\ref{def: 3D ring}) with a traveling speed $W>0$, we say a bounded open set $\Omg\subset\bbR^3$ is the \say{vortex domain} of $\overline\xi$ if it satisfies the following properties:

\medskip
(i) For the 3D flow map $\Phi_3(t,\cdot)$ generated by the relative vorticity $\xi$ given as  $\xi(t,(z,r)):=\overline\xi(z-Wt,r)$, we have 
$$\Phi_3(t,\Omg)=\set{(r,\tht,z+Wt)\in\bbR^3:(r,\tht,z)\in\Omg}\qd\mbox{for each}\q t\geq0.$$

(ii) Any proper superset $\tld{\Omg}\supsetneq\Omg$, which is open and bounded, fails to satisfy (i).

\medskip
\noindent Moreover, the set $\Omg-\overline{\set{\bfx\in\bbR^3:\overline\xi(\bfx)>0}}$ is called the \say{vortex atmosphere} of $\,\overline\xi$, which is the vortex domain minus the closure of the vortex core.
\end{defn}

\begin{defn}\label{def: 2D atmosphere}
For any vortex dipole (Definition~\ref{def: 2d dipole}) with a traveling speed $W>0$, we say a bounded open set $\Omg\subset\bbR^2$ is the \say{vortex domain} of $\overline\omg$ if it satisfies the following properties:

\medskip
(i) For the 2D flow map $\Phi_2(t,\cdot)$ generated by the vorticity $\omg$ given as $\omg(t,(x_1,x_2)):=\overline\omg(x_1-Wt,x_2)$, we have 
$$\Phi_2(t,\Omg)=\set{(x_1+Wt,x_2)\in\bbR^2:(x_1,x_2)\in\Omg}\qd\mbox{for each}\q t\geq0.$$

(ii) Any proper superset $\tld\Omg\supsetneq\Omg$, which is open and bounded, fails to satisfy (i).

\medskip
\noindent Moreover, the set $\Omg-\overline{\set{\bfx\in\bbR^2:\overline\omg(\bfx)\neq0}}$ is called the \say{vortex atmosphere} of $\,\overline\omg$, which is the vortex domain minus the closure of the vortex core.
\end{defn}

Later, in Section~\ref{sec: nonempty atmos.}, we present both nonempty and empty vortex atmosphere configurations.

\begin{rmk}[Uniqueness]\label{rmk: unique}
The uniqueness of a vortex domain follows directly from the maximality condition stated in (ii)-Definition~\ref{def: 3D atmosphere} and (ii)-Definition~\ref{def: 2D atmosphere}. Indeed, for two vortex domains $\Omg,\Omg'\subset\bbR^3$ corresponding to a vortex ring $\overline\xi$, it holds that
$$\Phi_3(t,\Omg\cup\Omg')=
\Phi_3(t,\Omg)\cup \Phi_3(t,\Omg')
\qd\mbox{for each}\q 
t\geq0.$$ The same conclusion holds for vortex dipoles.
\end{rmk} 

\begin{prop}[Existence]\label{prop: existence}
For any vortex ring (Definition~\ref{def: 3D ring}), there exists a unique vortex domain. Likewise, for any vortex dipole (Definition~\ref{def: 2d dipole}), there exists a unique vortex domain.
\end{prop}
\begin{proof}
The proof consists of four steps. First, we note that the vortex core satisfies (i) in Definition~\ref{def: 3D atmosphere} (or in Definition~\ref{def: 2D atmosphere}). Second, we show that there exists a bounded set containing every bounded open set satisfying (i). Third, we apply Zorn's lemma to obtain a maximal set. Lastly, the uniqueness follows from Remark~\ref{rmk: unique}.

\medskip
It suffices to explain the second step for the case of a vortex ring $\overline\xi$ in Definition~\ref{def: 3D ring} (for the case of a vortex dipole $\overline\omg$ can be proved similarly). As $\overline\xi$ is in $L^\ift(\bbR^3;\bbR)$ and has a compact support, the corresponding velocity $\overline\bfv(\bfx):=\bfv(0,\bfx)$ vanishes as $|\bfx|\to\ift$, where $\bfv:[0,\ift)\times\bbR^2\to\bbR^2$ is the 2D velocity generated by the traveling vortex ring $\xi(t,(z,r)):=\overline\xi(z-Wt,r)$ (e.g., see \cite[Lemma 1.1]{FT1981}). We will see that
there exists $R\gg 1$ such that
any bounded open set $\Omg\subset\bbR^3$ satisfying (i)-Definition~\ref{def: 2D atmosphere} must be contained in the bounded set $\set{(r,\tht,z)\in\bbR^3:r+|z|<R}$. In the remainder of the proof, we will use the usual cylindrical coordinate system of $\bbR^3$.

\medskip
We suppose the contrary, i.e., for each $n\geq1$, we suppose that there exists an axisymmetric bounded open set $\Omg_n\subset\bbR^3$ satisfying (i)-Definition~\ref{def: 3D atmosphere} and $$\Omg_n':=\Omg_n\cap\set{(r,\tht,z)\in\bbR^3:r+|z|>n}\neq\emptyset.$$ 
We take and fix some 
$n>2$ which is large enough such that $|\bfv(0,\cdot)|<W/10$ in the set $\set{(r,z,\tht)\in\bbR^3:r+|z|>n-2}$. Since $\Omg_n'$ is open, bounded, and axisymmetric, there exists a point $\bfx=(r,\tht,z)=\overline{\Omg_n'}$ satisfying $$\mbox{[Case I]:}\q\q
z=\inf\set{z'\in\bbR:(r',\tht',z')\in\overline{\Omg_n'}}<0\qd\mbox{or}\qd 
\mbox{[Case II]:}\q\q
z=\sup\set{z'\in\bbR:(r',\tht',z')\in\overline{\Omg_n'}}>0.
$$ For [Case I], let $\Phi_3=(\Phi_3^r,\Phi_3^\tht,\Phi_3^z)$ be the 3D flow map given by the ODE
$$\left\{\begin{aligned}
&\q\f{\rd}{\rd t}\Phi_3(t,\bfy)=
\bfv(t,\Phi_3(t,\bfy))\qd\mbox{for}\q t\geq0\\
&\q \Phi_3(0,\bfy)=\bfy\in\bbR^3.
\end{aligned}\right.$$ 
We observe that the function $t\mapsto \bfv(t,\Phi_3(t,\bfx))$ is continuous in $t\geq0$ and that $|\bfv(0,\Phi_3(0,\bfx))|=|\bfv(0,\bfx)|<W/10$. By continuity, there exists a constant $T>0$ such that $|\bfv(t,\Phi_3(t,\bfx))|<W/10$ for each $t\in[0,T]$. Hence we get $$|\Phi_3(t,\bfx)-\bfx|<\f{W}{10}t\qd\mbox{for each}\q t\in[0,T].$$ On the other hand, (i)-Definition~\ref{def: 3D atmosphere} says that $\Phi_3(t,\bfx)-(0,0,Wt)\in \overline{\Omg_n}$ for any $t\geq0$, in which two cases are possible: 
$$\mbox{[Case I-(1)]:}\q\q \Phi_3(t,\bfx)-(0,0,Wt)\in \overline{\Omg_n'}\qd\mbox{or}\qd\mbox{[Case I-(2)]}:\Phi_3(t,\bfx)-(0,0,Wt)\in \overline{\Omg_n}\setminus \overline{\Omg_n'}.$$
For [Case I-(1)], by the initial assumption on $\bfx=(r,\tht,z)$, we observe that \begin{equation}\label{eq:0901}
\Phi_3^z(t,\bfx)-Wt\geq z.
\end{equation} However, for each $t\in[0,T]$, we have 
$$z\geq \Phi_3^z(t,\bfx)-|\Phi_3^z(t,\bfx)-z|>\Phi_3^z(t,\bfx)-\f{W}{10}t,$$ which contradicts to the inequality \eqref{eq:0901}. For [Case I-(2)], we observe that 
$$\begin{aligned}
n&\geq \Phi_3^r(t,\bfx)+|\Phi^z_3(t,\bfx)-Wt|\\
&\geq r+|z-Wt|-2|\Phi_3(t,\bfx)-\bfx| \geq r+|z-Wt|-\f{Wt}{5}\\
&= r-z+Wt-\f{Wt}{5} > r-z
\end{aligned}$$ which means $r+|z|=r-z<n$. Note that it contradicts our initial assumption $\bfx=(r,\tht,z)\in\overline{\Omg_n'}$. In sum, [Case I] does not occur. 

\medskip
For [Case II], let $C_0\in(0,\ift)$ be any constant satisfying $$\sup_{t\geq0}\nrm{\bfv(t,\cdot)}_{L^\ift(\bbR^3)}<C_0,$$ for which one can apply the estimate (1.8) in \cite{LJ2025}:
$\nrm{\bfv}_{L^\ift}\lesssim \nrm{\overline\xi}_{L^\ift(\bbR^3)}^{1/2}\nrm{\overline\xi}_{L^1(\bbR^3)}^{1/4}\nrm{r^2\overline\xi}_{L^1(\bbR^3)}^{1/4}<\ift$. For the point $\bfx=(r,\tht,z)\in\overline{\Omg_n'}$ in [Case II], there exists the starting point $\bfx'=(r',\tht',z')\in\overline{\Omg_n}$ such that $$\Phi_3(\f{1}{C_0},\bfx')=\bfx+(0,0,\f{W}{C_0})$$ by (i)-Definition~\ref{def: 3D atmosphere}. For each $t\in[0,\f{1}{C_0}]$, we have 
$$\left|\Phi_3(t,\bfx')-[\bfx+(0,0,\f{W}{C_0})]\right|\leq \sup_{t\geq0}\nrm{\bfv(t,\cdot)}_{L^\ift(\bbR^3)}\cdot(\f{1}{C_0}-t)\leq C_0\cdot\f{1}{C_0}=1$$ and hence 
$$\begin{aligned}
\Phi^r_3(t,\bfx')+|\Phi^z_3(t,\bfx')|&\geq r+z+\f{W}{C_0}-2\left|\Phi_3(t,\bfx')-[\bfx+(0,0,\f{W}{C_0})]\right|\ \geq r+z-2 = r+|z|-2 > n-2, 
\end{aligned}$$ which implies that $\Phi_3(t,\bfx')\in\set{(r,z,\tht)\in\bbR^3:r+|z|>n-2}.$ Thus we get $|\bfv(t,\Phi_3(t,\bfx'))|<W/10$ for each $t\in[0,\f{1}{C_0}]$, which leads to 
\begin{equation}\label{eq0901_1}
\left|\bfx'-[\bfx+(0,0,\f{W}{C_0})]\right|=\left|\bfx'-\Phi_3(\f{1}{C_0},\bfx')\right|\leq 
\f{W}{10C_0}   
\end{equation} 
and therefore 
$$z'\geq z+\f{W}{C_0}-\left|
[\bfx'+(0,0,\f{W}{C_0})]\right|\geq 
z+\f{W}{C_0}-\f{W}{10C_0}>z.$$ From the initial definition on $\bfx=(r,\tht,z)$, the above relation $z'>z$ implies that $\bfx'=(r',\tht',z')\notin\overline{\Omg'_n}$ and so $r'+|z'|\leq n.$ However, by our previous observation \eqref{eq0901_1}, we get 
$$r'+|z'|\geq r+z+\f{W}{C_0}-2\left|\bfx'-[\bfx+(0,0,\f{W}{C_0})]\right|\geq r+z+\f{4W}{5C_0}=r+|z|+\f{4W}{5C_0}>n$$ which is a contradiction. It means that [Case II] does not occur. It completes the proof of Proposition~\ref{prop: existence}.
\end{proof}
By the above proposition(Proposition \ref{prop: existence}), the vortex atmosphere exists and is unique (up to a set of measure zero).
\subsection{Streamline}\label{sec: streamline}\q 

\medskip

Instead of treating a vortex ring or a vortex dipole as a traveling solution, we may view it as a stationary solution in the reference frame translating with constant speed $W>0$. In this moving frame, the closure of the vortex domain (core + atmosphere) can be characterized as the set of all bounded streamlines, or equivalently, as the collection of fluid particles whose trajectories remain bounded.

\medskip
For a vortex ring $\overline\xi$ and a point $(z',r')$, the streamline containing the point is defined as the 
 image in the $(z,r)$-plane $\Pi:=\{(z,r)\in\bbR^2:\,r>0\}$ of the trajectory map $\Phi_3(t,(z',r'))$ in the moving frame containing the point:
$$\left\{\Phi_3(t,(z',r'))-Wt\bfe_z\in\Pi\,\,:\,\,t\in\mathbb{R}\right\}\subset \Pi,$$ where we consider
the 3D axisymmetric flow map $\Phi_3=(\Phi_3^z,\Phi_3^r)$ generated by 
$\xi(t,(z,r)):=\overline\xi(z-Wt,r)$. This streamline is everywhere parallel to the local velocity in the moving frame, which is $\overline\bfv-W\bfe_z:=\bfv(0,\cdot)-W\bfe_z$, where $\bfv:[0,\ift)\times\bbR^3\to\bbR^3$ is the 3D velocity corresponding to the traveling vortex ring $\overline\xi(\cdot-Wt\bfe_z)$ with the unit vector $\bfe_z=(0,0,1)$ of the $z$-axis in the usual cylindrical coordinate system. That is, each streamline in the $(z,r)$-plane $\Pi$ is contained in some level set 
$$\Gmm_3(\gmm'):=\{(z,r)\in\Pi:\psi[\overline\xi]-Wr^2/2=\gmm'\}\qd\mbox{for some}\q\gmm'\in\bbR$$ where 
the 3D axisymmetric stream function $\psi=\psi[\overline\xi]$ satisfies the relation $$\overline\bfv=\overline{v}^z\bfe_z+\overline{v}^r\bfe_r, \qd \overline{v}^z=\f{\rd_r\psi[\overline\xi]}{r},\qd\overline{v}^r=\f{-\rd_z\psi[\overline\xi]}{r}$$ with the unit vector $\bfe_r=(\cos\tht,\sin\tht,0)$ of the $r$-axis in the cylindrical system (e.g., see \cite[Chapter 1]{FT1981} for a background).

\medskip
Similarly, for a vortex dipole $\overline\omg$, the velocity is given as $\overline\bfu-(W,0):=\bfu(0,\cdot)-(W,0)$ where $\bfu:[0,\ift)\times\bbR^2\to\bbR^2$ is the 2D velocity corresponding to the traveling vortex dipole $\overline\omg(\cdot-(Wt,0))$. 
Hence, the following level set 
$$\Gmm_2(\gmm'):=\{\bfx\in\bbR^2:x_2>0,\,\calG[\overline\omg]-Wx_2=\gmm'\},\qd\mbox{for}\q\gmm'\in\bbR$$ gives a streamline in the half plane $\set{\bfx\in\bbR^2:x_2>0}$, where the 2D stream function $\calG[\overline\omg]$ satisfies the relation
$$\overline\bfu=\left(\rd_{x_2}\calG[\overline\omg],-\rd_{x_1}\calG[\overline\omg]\right)$$ (e.g., see \cite[Chapter 1.2]{AC2022} for a background).

\medskip
To characterize (the closure of) a vortex domain as the collection of bounded streamlines, it is crucial to verify the boundedness of a level set containing a certain streamline. To this end, we observe the decays of $\psi[\overline\xi]/r^2$ and $\calG[\overline\omg]/x_2$ as $|\bfx|\to\ift$ in Proposition~\ref{prop: decay of stream/term} below. Its proof will be given in Appendix~\ref{appendix}.

\begin{prop}\label{prop: decay of stream/term}
For any axisymmetric $\overline\xi\in L^\ift(\bbR^3)$ having a compact support, its 3D axisymmetric stream function $\psi[\overline\xi]$ satisfies 
$$\sup_{|(z,r)|\geq R}\f{|\psi[\overline\xi](z,r)|}{r^2}\to0\qd\mbox{as}\q R\to\ift$$under the cylindrical coordinate system $(r,\tht,z)$ of $\bbR^3$.  Similarly, for any odd-symmetric $\overline\omg\in L^\ift(\bbR^2)$ having a compact support, its stream function $\calG[\overline\omg]$ satisfies 
$$\sup_{|(x_1,x_2)|\geq R}
\left|\f{\calG[\overline\omg](x_1,x_2)}{
x_2}\right|\to0
\qd\mbox{as}\q R\to\ift.$$
\end{prop}

\noindent By Proposition~\ref{prop: decay of stream/term} above, we have the following conclusions:
\begin{itemize}
\item For the case of $\gmm'>0$, the level sets $\Gmm_3(\gmm')$ and $\Gmm_2(\gmm')$ are bounded. Hence, every streamline in those level sets is bounded. 
\item The superlevel sets $\set{\psi[\overline\xi]-Wr^2/2>0}$ and $\set{\calG[\overline\omg]-Wx_2>0}\cap\set{x_2>0}$ are bounded. Note that they satisfy the conditions (i)-Definition~\ref{def: 3D atmosphere} and (i)-Definition~\ref{def: 2D atmosphere}, respectively. 
\item If $\gmm'<0$, we note that the level sets $\Gmm_3(\gmm')$ and $\Gmm_2(\gmm')$ are unbounded, so 
they 
$\cup_{\gmm'<0}\Gmm_3(\gmm')$ and $\cup_{\gmm'<0}\Gmm_2(\gmm')$
contain all possible unbounded streamlines. 
\end{itemize}

\noindent One might think that the bounded open sets $\set{\psi[\overline\xi]-Wr^2/2>0}$ and $\set{\calG[\overline\omg]-Wx_2>0}\cap\set{x_2>0}$ charaterize the vortex domains. It is true in many cases, but for some traveling solutions, an unbounded level set for some $\gmm'<0$ might contain more than two streamlines and one of them is  \textit{bounded} (see Figure~\ref{fig:pointring}). See Section~\ref{subsec: classif. of atmos.} for a detailed discussion of this observation.

\begin{figure}
\centering
\begin{subfigure}{0.45\textwidth}
\includegraphics[width=\linewidth]{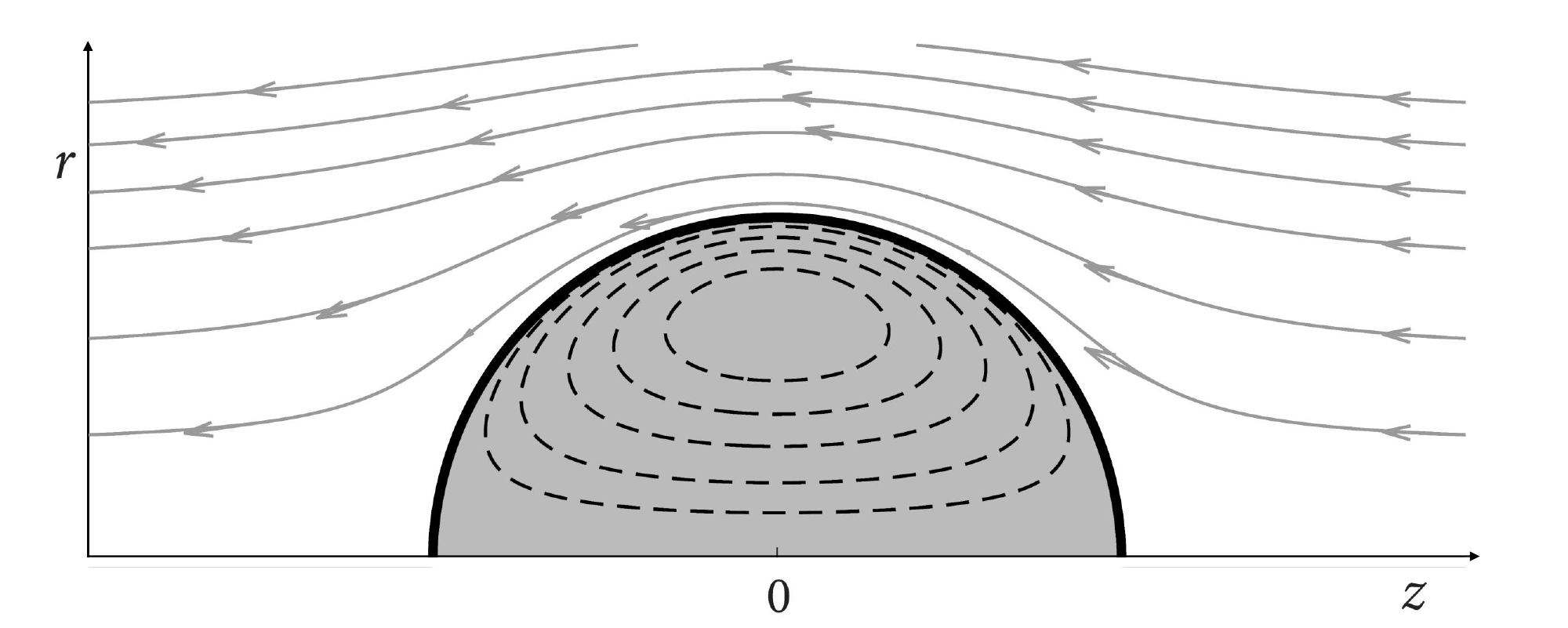}
\caption{Streamlines of Hill's spherical vortex in $\Pi$.}
\label{fig:Hillsvortex}
\end{subfigure}
\begin{subfigure}{0.45\textwidth}
\includegraphics[width=\linewidth]{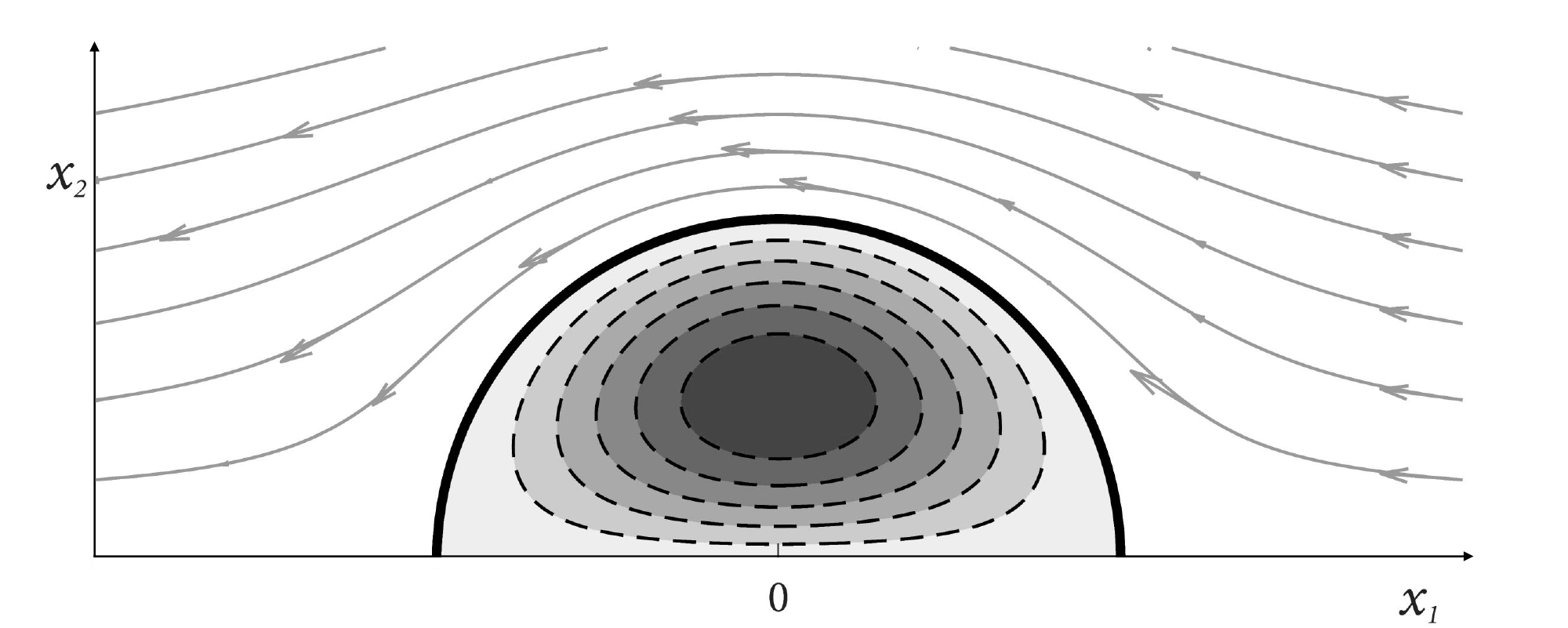}
\caption{Streamlines of Chaplygin--Lamb dipole in $\set{x_2>0}$.}
\label{fig:Lambdipole}
\end{subfigure}

\caption{
The figures above depict traveling vortices with empty atmospheres. The color shading indicates the (relative) vorticity values (darker means larger). The vortex domains coincide with the vortex cores, shown as the bounded regions enclosed by solid black curves. Dashed curves denote bounded streamlines, while arrows mark unbounded ones.
}
\label{fig1}
\end{figure}

\begin{figure}
\centering
\begin{subfigure}[t]{0.45\textwidth}
\includegraphics[width=\linewidth]{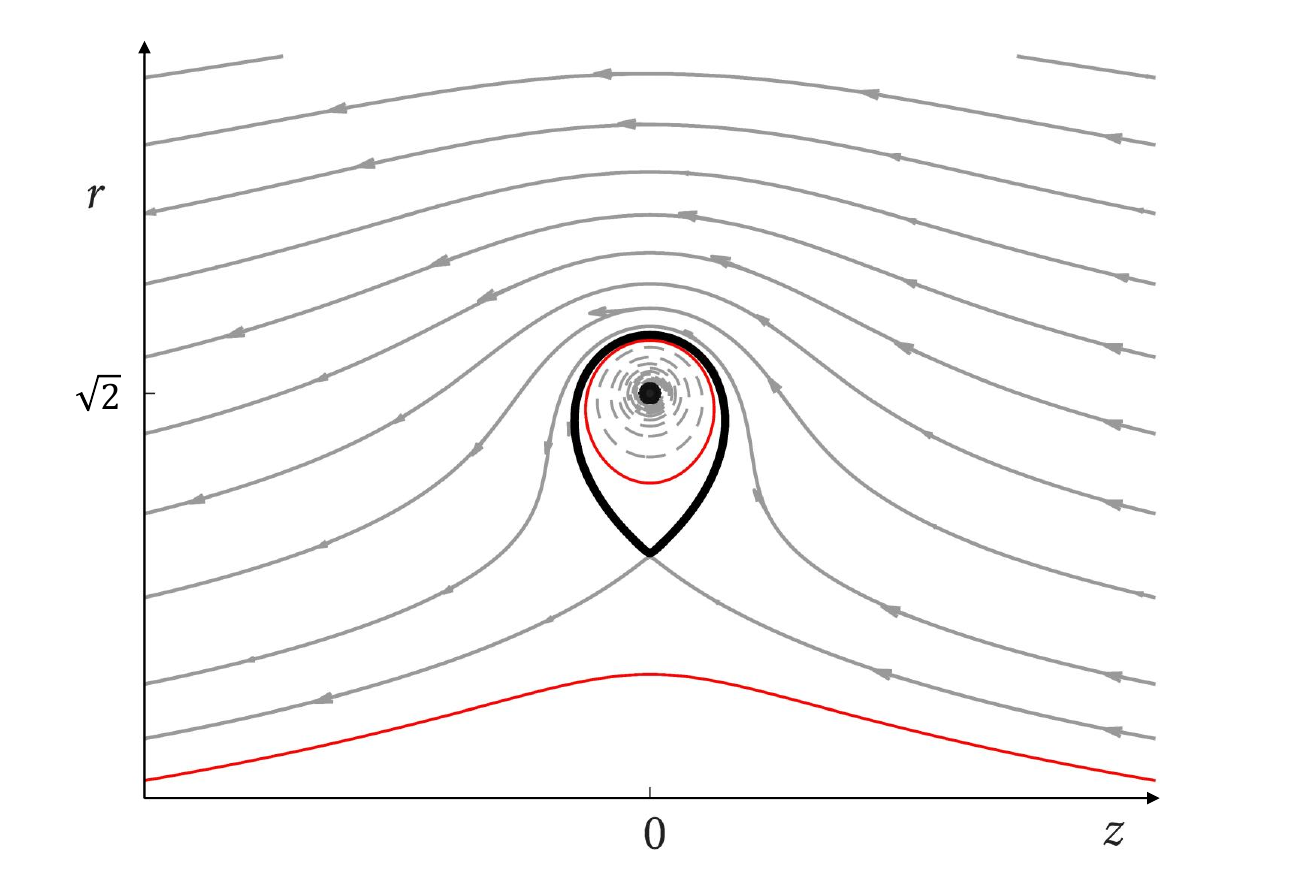}
\caption{Streamlines of a desingularized point-vortex ring plotted near $(0,\sqrt{2})\in\Pi$.  The red lines show a level set $\Gmm_3(\gmm')$ for some $\gmm'<0$ given by a union of a closed orbit and an unbounded streamline.
}
\label{fig:pointring}
\end{subfigure}
\begin{subfigure}[t]{0.45\textwidth}
\includegraphics[width=\linewidth]{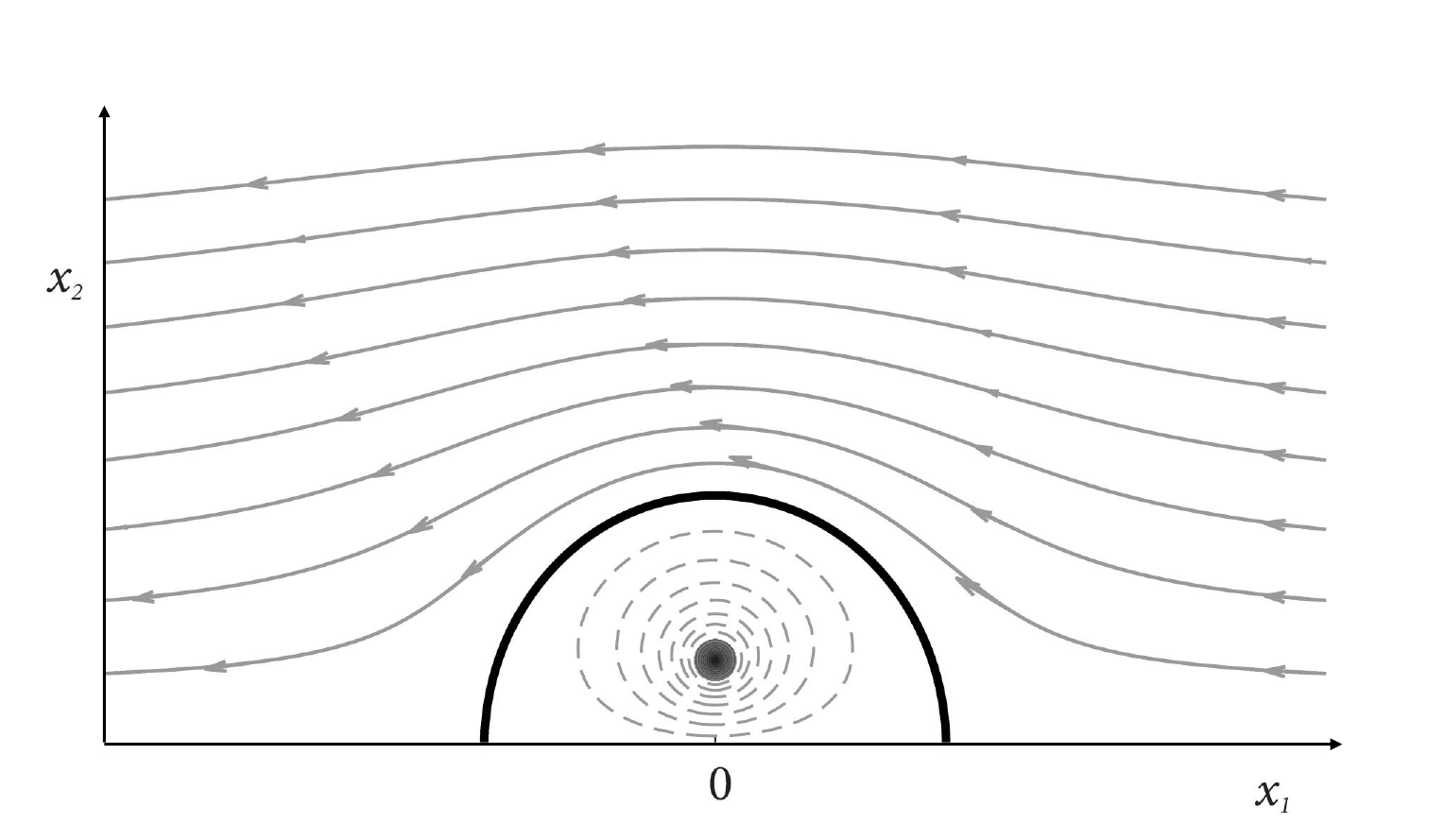}
\caption{Streamlines of  a point-vortex dipole in $\set{x_2>0}$.}
\label{fig:pointdipole}
\end{subfigure}
\caption{Examples of traveling solutions having nonempty atmospheres, especially the point-vortex cases, are shown. The supports of vortices, expressed by black dots, are concentrated around each point-vortex in $\Pi,\set{x_2>0}$ respectively, while the vortex domains (enclosed by black solid lines) surround the supports. In each figure, the atmosphere is the region between the support and the black solid line surrounding it.
}
\label{fig2}
\end{figure}

\subsection{Empty / non-empty atmosphere}\label{sec: nonempty atmos.}\q

\medskip
For cases in which the atmosphere is empty--that is, when the vortex domain coincides with its core--the Chaplygin–Lamb dipole \cite{Chap1903, Lamb1993} in 2D and Hill’s spherical vortex \cite{Hill1894} in 3D provide classical examples; see Figure~\ref{fig1}. An equivalent condition for a vortex dipole to have an empty atmosphere will be discussed in Section~\ref{sec: Sadovskii}.

\medskip

By contrast, consider a vortex ring $\overline\xi$ whose core is given by $\set{\psi[\overline\xi]-Wr^2/2-\gmm>0}$ in $\Pi$ or a vortex dipole $\overline\omg$ with its core $\set{\calG[\overline\omg]-Wx_2-\gmm>0}$ in $\set{x_2>0}$ for some \textit{positive} constant $\gmm>0$.  In either case, the vortex necessarily possesses a nonempty atmosphere, since there are   \textit{bounded} streamlines $\Gmm_3(\gmm'),\Gmm_2(\gmm')$ with $0<\gmm'<\gmm$ surrounding the core. Highly concentrated vortices, obtained from the desingularization of a point-vortex dipole and a point-vortex ring (see, e.g., \cite{FT1981, Turkington1983} and numerous related works), form typical examples for them; see Figure~\ref{fig2}. 
\medskip

Furthermore, each of Norbury’s steady vortex rings \cite{Norbury1973} also has a nonempty atmosphere: indeed, the vortex cores shown in \cite[Figure 3]{Norbury1973} are strictly contained within the vortex domains depicted in \cite[Figure 4]{Norbury1973}.

\section{Comparison of ring with dipole}\label{sec: compare ring and dipole}

We will prove that although the two-dimensional half-plane $\set{\bfx\in\bbR^2:x_2>0}$ and the  meridional $(z,r)$-plane $\Pi$ in three dimensions share certain structural properties, they nevertheless differ in several essential topological 
features of their vortex domains:

\medskip
\textit{In the plane  $\bbR^2$, the vortex domain of every vortex dipole must intersect the symmetry axis (as shown in Figure~\ref{fig:Lambdipole} and Figure~\ref{fig:pointdipole}). In contrast, in the axisymmetric setting of $\bbR^3$, the vortex domain of a vortex ring may be formulated away from the $z$-axis (as shown in Figure~\ref{fig:pointring}).}

\medskip

In other words, while a 2D dipole does not permit fluid particles to pass between its two poles, a 3D vortex ring may either block such particles or allow them to pass through its central hole.

\subsection{Steiner symmetry}\q

\medskip
Before we compare a ring with a dipole, we introduce the definition of Steiner symmetry that will be frequently used throughout the remainder of the paper.

\begin{defn}(2D Steiner symmetry about $x_2$-axis)\label{def: 2D steiner}
We say a function $\omg:\set{\bfx\in\bbR^2:x_2>0}\to[0,\ift)$ has the Steiner symmetry if, for each $x_2>0$, we have 
$\omg(x_1,x_2)=\omg(-x_1,x_2)$ for any  $x_1\geq0$ and that
$$x_1\mapsto\omg(x_1,x_2)\q\mbox{is non-increasing  in}\q x_1\geq0.$$ 
We say the symmetry is strict if the map is strictly decreasing.
Similarly, we say a set $A\subset\set{\bfx\in\bbR^2:x_2>0}$ has the Steiner symmetry if the characteristic function $\bfone_A:\set{\bfx\in\bbR^2:x_2>0}\to[0,\ift)$ has the Steiner symmetry. 
\end{defn}

\begin{defn}(3D Steiner symmetry about $z$-axis)\label{def: 3D steiner}
We say an axisymmetric function $\xi:\bbR^3\to[0,\ift)$ has the   Steiner symmetry if, for each $r>0$, we have 
$\xi(z,r)=\xi(-z,r)$ for any $z\geq0$ and that 
$$ z\mapsto \xi(z,r)\q\mbox{is non-increasing   in}\q z\geq0.$$
We say the symmetry is strict if the map is strictly decreasing.
Similarly, we say an axisymmetric set $A\subset\bbR^3$ has the Steiner symmetry if the characteristic function $\bfone_A:\bbR^3\to[0,\ift)$ has the Steiner symmetry. 
\end{defn}

\begin{rmk}\label{rmk: Steiner is natural} 
It is natural, from a physical viewpoint, to expect a traveling vortex (ring or dipole) to possess Steiner symmetry (Definitions~\ref{def: 2D steiner} and \ref{def: 3D steiner}). Many steady vortices that are known to have mathematical stability arise as maximizers of kinetic energy under suitable constraints; hence variational methods have been widely employed in their analysis (see, e.g., \cite{Arn1966, Ben1976, FT1981, Turkington1983, Burton1988, Bur1987, CQZZ2023, GS2024, BNL2013, AC2022, Wang2024, ACJ2025, Choi2024, CSW2025, CJS2025, AF1986}). In these variational settings, Steiner symmetry emerges as a necessary condition for energy maximization (see, for example, \cite[Theorem 3.9]{LL}).

\medskip

However, not all steady vortices arise from such variational principles. For instance, D. Cao et al.~\cite[Theorem 1.4]{CQZZ2021} constructed a two-dimensional vortex dipole whose  core\footnote{Nevertheless, the core of the example is simply connected.} is not Steiner symmetric, obtained via a bifurcation from the point-vortex dipole using an implicit function theorem applied to contour dynamics. Although the stability of this non-Steiner-symmetric dipole remains unknown, bifurcation methods of this type have proved effective in producing steady vortices with interesting topological structures. For a related discussion, including the existence of unequal-sized (asymmetric) vortex pairs, see \cite[Section~1]{HH2021}. 

\end{rmk} 

\subsection{Various types of vortex domain}\q

\medskip
For the rest of the paper, we only consider vortex rings or dipoles having the Steiner symmetry introduced in the previous subsection.

\subsubsection{Vortex domain of dipole}\q

\medskip
For any vortex dipole with the Steiner symmetry, we confirm that the vortex domain is 
touching the symmetry $x_1$-axis
 even when the core is not.\footnote{cf. When a core is touching the axis, we call it Sadovskii.   See \cite{CJS2025, CSW2025}.}

\begin{thm}\label{thm: 2D atmosphere}
For any vortex dipole $\overline\omg$ (Definition~\ref{def: 2d dipole}) having the Steiner symmetry (Definition~\ref{def: 2D steiner}) in the half plane $\set{\bfx\in\bbR^2:x_2>0}$, the corresponding vortex domain contains a disk
$\{\bfx\in\bbR^2:\,|\bfx|<R\}$ for some $R>0.$
\end{thm}

For the proof of Theorem~\ref{thm: 2D atmosphere} above, we establish a proposition asserting that the fluid speed at the dipole’s center exceeds twice the traveling speed. For earlier discussion concerning the conclusion of Theorem~\ref{thm: 2D atmosphere}, we refer the reader to \cite[Subsection~1.4]{CJS2025}.

\begin{prop}\label{prop: 2D central speed}
For every vortex dipole $\overline\omg$ (Definition~\ref{def: 2d dipole}) with a traveling speed $W>0$, if it has the Steiner symmetry (Definition~\ref{def: 2D steiner}) in the half plane $\set{\bfx\in\bbR^2:x_2>0}$, we have 
$$\overline u^1(0,0)>2W$$ where $\overline u^1$ is the first component of the velocity $\overline\bfu:=\nb^\perp\calG[\overline\omg]$.
\end{prop}

\begin{proof}
The proof is essentially contained in the proof of \cite[Lemma 4.5]{CJS2025} for the case of a patch dipole $\omg=\bfone_{A_+}-\bfone_{A_-}$. For a general case, we see \cite[Proposition 4.20]{CSW2025}. See Figure~\ref{fig: dipole center} for a simple illustration.
\end{proof}

We now prove Theorem~\ref{thm: 2D atmosphere} using Proposition~\ref{prop: 2D central speed} above. 

\begin{proof}[\textbf{Proof of Theorem~\ref{thm: 2D atmosphere}}] 
From the odd-symmetry of vortex dipole $\overline\omg$, we confine our focus to the half plane $\set{x_2>0}$. As the set $\Omg:=\set{\calG[\overline\omg]-Wx_2>0}\cap\set{x_2>0}$ is bounded (due to the decay of $\calG[\overline\omg]/x_2$ given in Proposition~\ref{prop: decay of stream/term}) and consists of bounded streamlines consisting/surrounding the traveling vortex core, it is contained in the vortex domain intersecting the half plane $\set{x_2>0}$. Hence, for the proof of Theorem~\ref{thm: 2D atmosphere}, it suffices to show that $\Omg$ touches the $x$-axis, i.e.,
$$\{\bfx\in\set{x_2>0}:\,|\bfx|<R\}\subset 
\{\bfx\in\set{x_2>0}:\calG[\overline\omg](\bfx)-Wx_2>0\}
\qd\mbox{for some}\q R>0.$$ Knowing that $\calG[\overline\omg]\equiv0$ in the $x_1$-axis,
the above relation can be proved by showing that the relation $$\rd_{x_2}\calG[\overline\omg](0,0)>W$$ holds for every vortex dipole $\overline\omg$ in Definition~\ref{def: 2d dipole}. We complete the proof by noting that the above relation is given in Proposition~\ref{prop: 2D central speed}, with the right-hand side replaced by $2W$.
\end{proof}

\subsubsection{Vortex domain of ring}\q

\medskip

For vortex rings, however, both configurations may occur: a vortex domain of spheroidal type (Theorem~\ref{thm: 3D sphere}, Figure~\ref{fig:Hillsvortex}) and one of toroidal type (Theorem~\ref{thm: 3D donut}, Figure~\ref{fig:pointring}). Unlike the former, the latter configuration allows certain fluid particles located ahead of the translating ring to pass through its central opening.

\begin{thm}\label{thm: 3D sphere}
There exists a vortex ring $\overline\xi$ (Definition~\ref{def: 3D ring}) having the Steiner symmetry (Definition~\ref{def: 3D steiner})
such that the corresponding vortex domain contains an open ball $\set{\bfx\in\bbR^3:|\bfx|<R}$ for some $R>0$.
\end{thm}
\begin{proof}
  Hill's spherical vortex $\xi_H=\bfone_{\set{|\bfx|<1}}$ is the classical example of a spheroidal vortex domain. A broader family of such domains appears in the work of Norbury \cite{Norbury1973}, where the author gives a numerical description of a one-parameter family, $\sqrt{2}\geq\alp\geq0.1$, of the vortex rings extending from   Hill's spherical vortex ($\alp=\sqrt2$) to vortex rings of small cross-section ($\alp\to0$); see \cite[Figure 4]{Norbury1973} for the corresponding vortex domains. 
\end{proof} 

\begin{thm}\label{thm: 3D donut} 
There exists a vortex ring $\overline\xi$ (Definition~\ref{def: 3D ring}) having the Steiner symmetry (Definition~\ref{def: 3D steiner})
such that the corresponding vortex domain is away from the $z$-axis (as shown in Figure~\ref{fig:pointring}), in the sense that the vortex domain is contained in the set $\{(r,\tht,z)\in\bbR^3:\,r>L\}$ for some $L>0$. 
\end{thm}
\begin{proof}
We can find an example in the paper \cite{FT1981}, in which a vortex ring $\xi_\lmb$ with the vorticity function $f(s)=\lmb\bfone_{\set{s>0}}$ was constructed for each $\lmb>0$. \cite[Theorem 2.1]{FT1981} establishes the existence of an axisymmetric vortex ring for each $\lmb>0$: $$\xi_\lmb=\lmb\cdot\bfone_{\set{\psi-W r^2/2-\gmm>0}}\qd\mbox{for some constants}\q W>0,\gmm\geq0$$ (with the 3D axisymmetric stream function $\psi=\psi[\xi_\lmb]$) having the Steiner symmetry and other assumptions in Definition~\ref{def: 3D ring}, together with 
$$\f{1}{2}\int_{\bbR^3}r^2\xi_\lmb\,\dd\bfx=1
\qd\mbox{and}\qd
\int_{\bbR^3}\xi_\lmb\,\dd\bfx\leq1.$$  

\noindent\textbf{Step I. Asymptotic behavior of $\xi_\lmb$ as $\lmb\to\ift$}\q

As $\lmb\to\ift$, the function $2\pi r\xi(z,r)$ approaches the dirac-delta $\dlt(\cdot-(0,\sqrt{2}))$ in $\Pi$, as shown in Figure~\ref{fig:pointring}. Indeed, by Remark~{1} and Theorem~{6.5} in \cite{FT1981}, we have 
$\int_{\bbR^3}\xi_\lmb\,\dd\bfx=1$ for $\lmb\gg1$. Moreover, \cite[Lemma 7.4]{FT1981} tells us that 
$$diam(\supp\xi_\lmb)\lesssim \lmb^{-1/2}\qd\mbox{for}\q \lmb\gg1$$ where $diam(\supp\xi_\lmb)$ denotes the diameter of the cross-section of the axisymmetric set $\supp\xi_\lmb$ in the $(z,r)$-plane $\Pi$, while 
$$\inf\set{r>0:(z,r)\in\supp\xi_\lmb}\leq \sqrt{2}\leq \sup\set{r>0:(z,r)\in\supp\xi_\lmb}$$ due to the constraint $\f{1}{2}\int_{\bbR^3}r^2\xi_\lmb\,\dd\bfx=1$.\\

\noindent\textbf{Step II. Strict Steiner symmetry of the stream function $\psi$}\q

For each $r>0$, the function $\psi(z,r)$ decreases strictly in $z\geq0$.  Indeed, for each $z>0$, we observe that
$$\rd_z\psi=
\int_\Pi\rd_z G(r,z,r',z')\xi_\lmb(z',r')r'\dd r'\dd z'$$ with 
$$\rd_z G=-(z-z')\f{rr'}{2\pi}\ii{0}{\pi}\cos\vartheta\left[r^2+r'^2-2rr'\cos\vartheta+(z-z')^2\right]^{-3/2}\dd\vartheta$$
(e.g., see \cite[Chapter 1]{FT1981} for a background)
which gives
$$\begin{aligned}
\rd_z\psi&=\int_\Pi
\left[\ii{0}{\pi}\f{\cos\vartheta}{\left[r^2+r'^2-2rr'\cos\vartheta+z'^2\right]^{3/2}}\dd\vartheta
\right]\f{rr'}{2\pi}
z'\xi_\lmb(z'+z,r')r'\dd r'\dd z'\\
&=\int_\Pi\underbrace{\left[\ii{0}{\pi/2}
\f{\cos\vartheta}{\left[r^2+r'^2-2rr'\cos\vartheta+z'^2\right]^{3/2}}-\f{\cos\vartheta}{\left[r^2+r'^2+2rr'\cos\vartheta+z'^2\right]^{3/2}}\,\dd\vartheta
\right]}_{>0}\f{rr'}{2\pi}
z'\xi_\lmb(z'+z,r')r'\dd r'\dd z' \\
&=\left(\int_{\set{z'>0}\cap\Pi}+\int_{\set{z'<0}\cap\Pi}\right)\left[\ii{0}{\pi/2}\cdots\dd\vartheta\right]\f{rr'}{2\pi}
z'\xi_\lmb(z'+z,r')r'\dd r'\dd z'\\
&=\int_{\set{z'>0}\cap\Pi}
\left[
\ii{0}{\pi/2}\cdots\dd\vartheta\right]\f{rr'}{2\pi}
z'\underbrace{\left(\xi_\lmb(z'+z,r')-\xi_\lmb(-z'+z,r')\right)}_{\leq0}
r'\dd r'\dd z'<0.
\end{aligned}$$ Note that we only used the Steiner symmetry of $\xi_\lmb$, and so the relation \say{$\rd_z\psi<0$ in $\set{z>0}$} holds for any vortex ring having the Steiner symmetry. \\

\noindent\textbf{Step III. Bound of $z$-velocity}\q

For large enough $\lmb\gg1,$ it holds that
$$v^z(z,r)<\f{W}{2}
\qd\mbox{for each}\q 0\leq r\leq 1,\, z\in\bbR.$$ 
To show this,
we first fix $z\in\bbR$, and observe that
$$\begin{aligned}
v^z(z,r)=\f{1}{r}\rd_r\psi&=\int_\Pi
\f{r'}{2\pi r}\left(\ii{0}{\pi}\f{\cos\vartheta\left(
r'^2-rr'\cos\vartheta+(z-z')^2
\right)}{\left[r^2+r'^2-2rr'\cos\vartheta+(z-z')^2\right]^{3/2}}\dd\vartheta
\right)\xi_\lmb(z',r')r'\dd r'\dd z'\\
&=\int_\Pi\f{r'}{2\pi}
\left(\f{1}{r}\ii{0}{\pi/2}\cos\vartheta\left[f(\cos\vartheta)-f(-\cos\vartheta)\right]
\dd\vartheta
\right)\xi_\lmb(z',r')r'\dd r'\dd z'\\
&=\int_\Pi\f{r'}{2\pi}
\left(\f{1}{r}
\ii{0}{\pi/2}
2(\cos\vartheta)^2 f'(s)
\dd\vartheta
\right)\xi_\lmb(z',r')r'\dd r'\dd z'
\end{aligned}$$ for some $s=s(\vartheta,r,r',z,z')\in(-\cos\vartheta,\cos\vartheta)$, where 
$$f(s)=\f{\left(
r'^2-rr's+(z-z')^2
\right)}{\left[r^2+r'^2-2rr's+(z-z')^2\right]^{3/2}}$$ and 
$$f'(s)=r\left(
-r'D+3r'\left(r'^2-rr's+(z-z')^2\right)
\right)D^{-5/2}\qd\mbox{with}\q
D=r^2+r'^2-2rr's+(z-z')^2.$$
That is, we get 
$$
v^z(z,r)
=\int_\Pi\f{r'}{2\pi}\left(
\ii{0}{\pi/2}
2(\cos\vartheta)^2\left(
-r'D+3r'\left(r'^2-rr's+(z-z')^2\right)
\right)D^{-5/2}\dd\vartheta
\right)\xi_\lmb(z',r')r'\dd r'\dd z'.$$ As we have
$$ (r-r')^2+(z-z')^2\leq D\leq (r+r')^2+(z-z')^2,$$ we get
$$\begin{aligned}
|v^z(z,r)|&\lesssim
\int_\Pi r'^2\left(
\ii{0}{\pi/2}(D+r'^2+rr'+(z-z')^2)D^{-5/2}\dd\vartheta
\right)\xi_\lmb(z',r')r'\dd r'\dd z'\\
&\lesssim
\int_\Pi r'^2\f{(r+r')^2+(z-z')^2}{[(r-r')^2+(z-z')^2]^{5/2}}\xi_\lmb(z',r')r'\dd r'\dd z'.
\end{aligned}$$
As $\lmb\to\ift$, we have $r'\to\sqrt2$, and by assuming $r<1$, we get 
$$|v^z(z,r)|\lesssim
\int_\Pi\f{1}{(1+(z-z')^2)^{3/2}}\xi_\lmb(z',r')r'\dd r'\dd z'\lesssim 1$$ for sufficiently large $\lmb>0$. Since $W\gtrsim\log\lmb$ as $\lmb\gg1$ due to Theorem~{6.5} and Lemma~{6.1} in \cite{FT1981}, we have $$|v^z(z,r)|\leq \f{W}{2}\qd\mbox{for each}\q 
0\leq r\leq 1,\, z\in\bbR,\qd\mbox{given that}\q\lmb\gg1.$$

\noindent\textbf{Step IV. Unbounded streamlines below the vortex core}\q

It remains to show that there exists an unbounded streamline  of the form $\set{(z,l(z))\in\Pi:z\in\bbR}$ for some function $l:\bbR\to(0,1)$
that is in $C^1(\bbR)$ and passes under the vortex core $\set{\xi_\lmb>0}$. It would imply that the 3D flow map $\Phi_3(t,\cdot)$ starting below the streamline $\set{r<l(z)}$ stays beneath the line (in the moving frame) for all time, i.e.,
$$\set{\Phi_3(t,(z,r))-Wt\bfe_z:t\in\bbR,\,r<l(z)}\subset\set{r<l(z)},$$ while such a particle fails to hop on the moving frame $(W,0)$ due to the bound of $z$-velocity $v^z$ in \textbf{Step III}, which is $W/2$.\\

\noindent\textit{(Claim) If $\lmb\gg1$ is large enough, there exists an unbounded streamline $\set{(z,l(z))\in\Pi:z\in\bbR}$ for some function $l:\bbR\to(0,1)$
that is in $C^1(\bbR)$, contained in the level set $\Gmm_3(\gmm')=\set{\psi-W r^2/2=\gmm'}$ for some $\gmm'<0$,
such that $$l(z)<\inf\set{r>0:\, (z,r)\in\supp\xi_\lmb}\qd\mbox{for each}\q z\in\bbR.$$}

Let $\lmb\gg1$ be large enough.  Then we have 
$$\lim_{r\to0}\left(\f{\psi(0,r)}{r^2}-\f{W}{2}\right)=\f{1}{2}\left(v^z(0,0)-W\right)\leq -\f{W}{4}\lesssim -\log\lmb<0$$ due to \textbf{Step III}. Then, we fix a small $L\ll1$ such that 
$$\f{\psi(0,L)}{L^2}-\f{W}{2}<0.$$ We then put 
$$\phi(z,r):=\psi(z,r)-\f{W}{2}r^2 
\qd\mbox{and}\qd
\gmm':=\phi(0,L)<0.$$ By \textbf{Step III}, we observe that for each $z>0$, the function $\phi(z,r)$ decreases strictly in $r\in[0,L]$ due to
$$\rd_r\phi(z,r)=\rd_r\psi-W r=r(v^z(z,r)-W)<0.$$ Moreover, for each $z>0$, we have 
$$\phi(z,0)=0>\gmm'\qd\mbox{and}\qd
\phi(z,L)<\phi(0,L)=\gmm'$$ by \textbf{Step II}, and therefore there exists a unique constant $l=l(z)\in(0,L)$ such that $$\phi(z,l(z))=\gmm'.$$ The same holds for $z<0$, with the relation $l(-z)=l(z)$ due to the even symmetry of $\psi.$ We put $l(0):=L$, and we obtain a streamline $\calS$ parametrized by $l:\bbR\to(0,L]$, i.e.,
$$\calS:=\set{(z,l(z))\in\Pi:z\in\bbR}$$ contained in the level set 
$$\Gmm_3(\gmm')=\set{\psi-W r^2/2=\gmm'}.$$ We note that $\calS$ pass below the vortex core $\supp\xi_\lmb$ since $\lmb\gg1$ and $L\ll1$.

We can easily show that $l$ is continuous. Since $\rd_r\phi(z,r)<0$ for each $r\in(0,L),\,z\in\bbR$, we apply the implicit function theorem at each point $(z,l(z))$ to verify that the extended streamline $l:\bbR\to(0,L]$ is locally $C^1$. One can verify  $l\in C^1(\bbR)$ by the following relation 
$$\psi(z,l(z))-W\f{l(z)^2}{2}=\gmm',$$ which leads to 
$$l'(z)=\f{v^r(z,l(z))}{v^z(z,l(z))-W}\q\imp\q
|l'(z)|\leq \f{2\nrm{\bfv}_{L^\ift}}{W}<\ift$$ due to \textbf{Step III} and the estimate (1.8) in \cite{LJ2025}:
$$\nrm{\bfv}_{L^\ift}\lesssim \nrm{\xi}_{L^\ift(\bbR^3)}^{1/2}\nrm{\xi}_{L^1(\bbR^3)}^{1/4}\nrm{r^2\xi}_{L^1(\bbR^3)}^{1/4}\lesssim \lmb^{1/2}.$$
It completes the proof of \textit{(Claim)}, and so Theorem~\ref{thm: 3D donut} is proved.
\end{proof}

\section{Simply-connected vortex domain}\label{subsec: classif. of atmos.}

In this section, we extend our analysis by incorporating an additional geometric condition beyond Steiner symmetry, namely \textit{simply-connectedness}. In particular, we show that if a vortex core is simply connected, then its vortex domain is also simply connected. This follows from the characterization of a vortex domain as a superlevel set of the associated stream function. Indeed, the examples used in the proofs of Theorems~\ref{thm: 3D sphere} and~\ref{thm: 3D donut} are all simply connected, and hence their vortex domains arise precisely as superlevel sets.

\medskip

Throughout this section, several arguments appear in both two and three dimensions with only minor modifications—for example, Lemmas~\ref{lem: 2D core in superlevel set} and~\ref{lem: 3D core in superlevel set}, and Lemmas~\ref{lem: 2D streamline formation} and~\ref{lem: 3D streamline formation}. Since one of the aims of this paper is to provide a careful, side-by-side analysis of the similarities and differences between two-dimensional vortex dipoles and three-dimensional vortex rings, we present the parallel arguments explicitly in order to keep the exposition complete and self-contained.

\medskip
 From now on, particle trajectories are understood in the moving frame.

\subsection{Dipole with simply-connected core}\q

\medskip
We begin by considering the case of a 2D vortex dipole. We give the main statement below, whose proof will be given at the end of this subsection. 
\begin{thm}\label{thm: class. of 2D atmos.}
Let $\overline\omg$ be any vortex dipole (Definition~\ref{def: 2d dipole}) with a traveling speed $W>0$. In the half plane $\set{\bfx\in\bbR^2:x_2>0}$, if $\overline\omg$ has the Steiner symmetry (Definition~\ref{def: 2D steiner}) with a simply-connected vortex core, then the corresponding vortex domain $\Omg$ also has the Steiner symmetry and is simply connected in the half plane. In particular, the vortex domain is given by the superlevel set $$\Omg=\left\{\bfx\in\bbR^2:x_2>0,\,\calG[\overline\omg](\bfx)-Wx_2>0\right\}$$ in the half plane. Moreover, the vortex domain has an \say{oval-shape} in the sense that, for some constant $R\in(0,\ift)$ and a function $l:[0,R]\to[0,\ift)$ satisfying $l>0$ in $[0,R)$, $l(R)=0$, and $l\in C([0,R])\cap C^1((0,R))$, we have
$$\Omg=\left\{\bfx\in\bbR^2:
x_2\in(0,R),\,|x_1|<l(x_2)\right\}.$$
\end{thm}

Lemma~\ref{lem: 2D streamline formation} and Lemma~\ref{lem: 2D core in superlevel set} below will be used in the proof of Theorem~\ref{thm: class. of 2D atmos.} that follows.

\begin{lem}\label{lem: 2D streamline formation}
Let $\overline\omg$ be any vortex dipole (Definition~\ref{def: 2d dipole}) with a traveling speed $W>0$. In the half plane $\set{\bfx\in\bbR^2:x_2>0}$, if $\overline\omg$ has the Steiner symmetry (Definition~\ref{def: 2D steiner}), then its stream function $\calG[\overline\omg]$ has the \textit{strict} Steiner symmetry in the half plane. Moreover, any point $\bfy=(y_1,y_2)\in\set{\bfx\in\bbR^2:x_1<0,x_2>0}$ satisfying $$\calG[\overline\omg](y_1,y_2)-Wy_2=:\gmm<0\qd\mbox{and}\qd 
\calG[\overline\omg](0,s)-Ws>\gmm\qd\mbox{for each}\q 
s\in(-\gmm/W,y_2)
$$ is not contained in the corresponding vortex domain. See Figure~\ref{fig: 2D streamline formation} for a simple illustration for this.
\end{lem}
\begin{figure}
\centering
\includegraphics[width=0.9\linewidth]{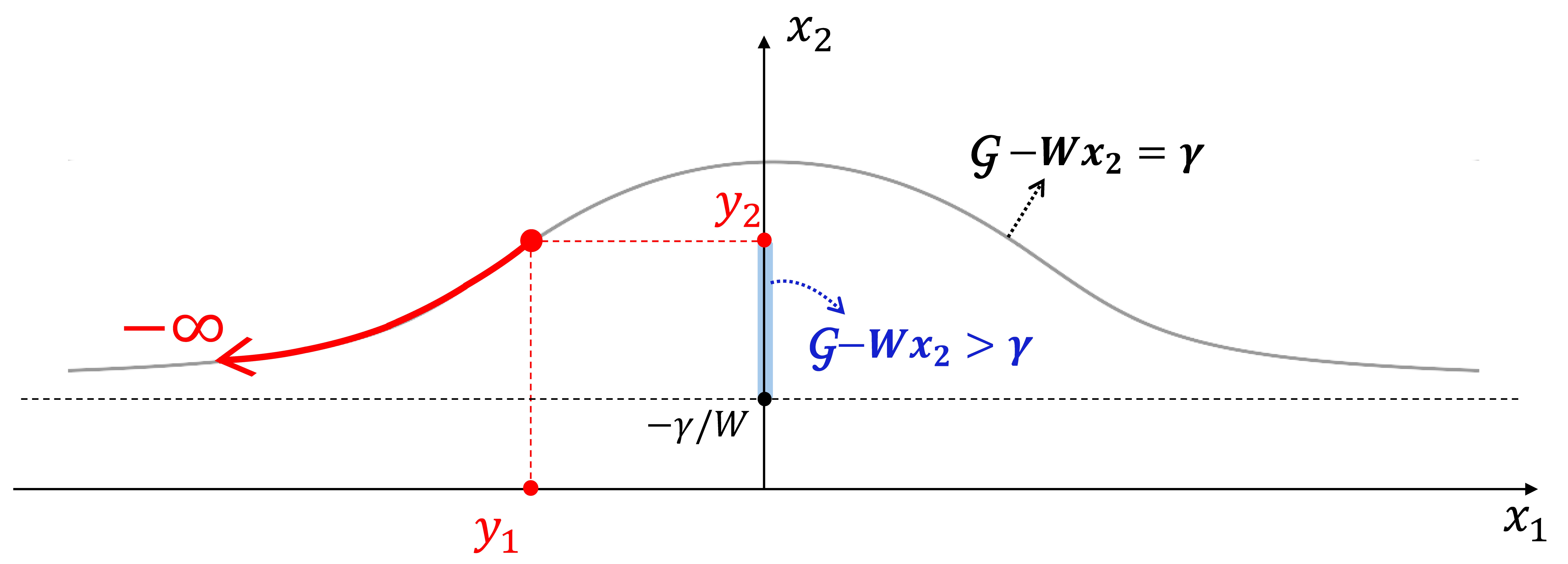}
\caption{In Lemma~\ref{lem: 2D streamline formation}, the flow map takes the particle at $\bfy=(y_1,y_2)$ indefinitely away from the origin. For all time, this particle is imprisoned in the level set $\set{\calG[\overline\omg]-Wx_2=\gmm}$ and asymptotically approaches the horizontal line $\set{x_2=-\gmm/W}$.} 
\label{fig: 2D streamline formation}
\end{figure}

\begin{proof}
We first note that the stream function $\calG=\calG[\overline\omg]$ is given by 
\begin{equation}\label{eq: stream}
\calG(\bfx)=
\int_{\bbR^2}\f{1}{2\pi}\left(-\log|\bfx-\bfy|\right)\overline\omg(\bfy) d\bfy
=\int_{y_2>0}\f{1}{2\pi}\left(\log|\bfx-\bfy^*|-\log|\bfx-\bfy|
\right)\overline\omg(\bfy) d\bfy
\end{equation}
where $\bfy^*=(y_1,-y_2)$ and satisfies 
\begin{equation}\label{eq1125_3}
\calG(x_1,x_2)=-\calG(x_1,-x_2)
\qd\mbox{and}\qd\calG(x_1,0)=0
\end{equation}
for any $(x_1,x_2)\in\bbR^2$, due to the odd symmetry of $\overline\omg$. Moreover, the even symmetry $\calG(x_1,x_2)=\calG(-x_1,x_2)>0$ is obvious, and for fixed $x_1>0$, we have
\begin{equation}\label{eq1117_1}\begin{aligned}
2\pi\rd_{x_1}\calG(x_1,x_2)&=\int_{y_2>0}(y_1-x_1)\left[\f{1}{|\bfx-\bfy|^2}-\f{1}{|\bfx-\bfy^*|^2}\right]\overline\omg(\bfy) \dd \bfy =\int_{y_2>0}\f{4x_2y_2\q(y_1-x_1)}{|\bfx-\bfy|^2|\bfx-\bfy^*|^2}\overline\omg(\bfy) \dd\bfy\\
&=\ii{0}{\ift}\ii{-\ift}{\ift} \f{4x_2y_2y_1}{(y_1^2+(x_2-y_2)^2)(y_1^2+(x_2+y_2)^2)}\overline\omg(y_1+x_1,y_2) \dd y_1\dd y_2\\
&=\ii{0}{\ift}\ii{0}{\ift}\f{4x_2y_2y_1}{(y_1^2+(x_2-y_2)^2)(y_1^2+(x_2+y_2)^2)}\underbrace{\left[\overline\omg(y_1+x_1,y_2)-\overline\omg(-y_1+x_1,y_2)\right]}_{\leq0} \dd y_1 \dd y_2\\
&<0
\end{aligned}\end{equation} by the Steiner symmetry of $\overline\omg$. It implies that $\calG$ has the \textit{strict} Steiner symmetry. 

\medskip Now, we suppose that there exists a point $\bfy=(y_1,y_2)\in\set{\bfx\in\bbR^2:x_1<0,x_2>0}$ satisfying 
$$\calG[\overline\omg](y_1,y_2)-Wy_2=:\gmm<0\qd\mbox{and}\qd 
\calG[\overline\omg](0,s)-Ws>\gmm\qd\mbox{for each}\q 
s\in(-\gmm/W,y_2).$$
We will show that the negativity of $\gmm$ prevents the point $\bfy$ from being contained in the core. To this end, we will construct an unbounded streamline along which the particle initially at $\bfy$ moves away from the origin in the moving frame indefinitely as $t\to\ift.$ 

\medskip
For each $x_2\in(-\gmm/W,y_2]$, we consider the equation
\begin{equation}\label{eq0922_1}
\calG(x_1,x_2)-Wx_2=\gmm\qd\mbox{for}\q x_1<0.
\end{equation} Since $\calG(0,x_2)-Wx_2>\gmm$ for each $x_2\in(-\gmm/W,y_2]$, the \textit{strict} Steiner symmetry of $\calG$ guarantees that the solution of \eqref{eq0922_1} is uniquely determined by a continuous function $x_1=-f(x_2)$ for some $f:(-\gmm/W,y_2]\to(0,\ift)$ with $f(y_2)=-y_1$, whose differentiability is obvious from the differentiability and the \textit{strict} Steiner symmetry of $\calG$. We now consider the limit $x_2\downarrow-\gmm/W$. Since $\calG>0$ in $\set{x_2>0}$ and $\calG\to0$ as $|\bfx|\to\ift$, 
we have $f(x_2)\to\ift$ as $x_2\downarrow-\gmm/W$. Otherwise, if $\liminf_{x_2\downarrow -\gmm/W} f(x_2)\in[0,\ift)$, then for some subsequence $s_n\to -\gmm/W$, we have 
$$0=\lim_{n\to\ift}Ws_n+\gmm=\lim_{n\to\ift}\calG(-f(s_n),s_n)=\calG\left(-\liminf_{x_2\downarrow -\gmm/W} f(x_2), -\gmm/W\right)>0,$$ which is a contradiction.

\medskip
We now consider the trajectory of a particle initially at $\bfy$. As its streamline must be contained in the level set of $\calG-Wx_2$, and
the vertical velocity $\overline{u}^2=-\rd_{x_1}\calG$ is negative in the left side $\set{\bfx\in\bbR^2:x_1<0,x_2>0}$ (by \eqref{eq1117_1}) the particle must travel along the curve $(-f(x_2),x_2)$ starting from $x_2=y_2$ and heading to $x_2\downarrow-\gmm/W$ without any standstill; see Figure~\ref{fig: 2D streamline formation}.
Indeed,  if the particle cannot go below than $-\gmm/W+\epsilon$ forever for some $\epsilon>0$, then we have a contradiction due to the fact $$\sup_{s\in[-\gmm/W+\epsilon,y_2]} \overline{u}^2(-f(s),s)<0.$$
Thus, the particle starting from $\bfy$ escapes and gets indefinitely far away from the origin. It completes the proof.\end{proof}

\begin{lem}\label{lem: 2D core in superlevel set}
Let $\overline\omg$ be any vortex dipole (Definition~\ref{def: 2d dipole}) with a traveling speed $W>0$. In the half plane $\set{\bfx\in\bbR^2:x_2>0}$, if $\overline\omg$ has the Steiner symmetry (Definition~\ref{def: 2D steiner}) with a simply-connected vortex core $\set{\bfx\in\bbR^2:x_2>0,\,\overline\omg(\bfx)>0}$, then the superlevel set 
$$\left\{\bfx\in\bbR^2:x_2>0,\,\calG[\overline\omg](\bfx)-Wx_2>0\right\}$$ also has the Steiner symmetry (Definition~\ref{def: 2D steiner}), is simply connected, and contains the vortex core $\set{\bfx\in\bbR^2:x_2>0,\,\overline\omg(\bfx)>0}$.
\end{lem}
\begin{proof}
Lemma~\ref{lem: 2D streamline formation} says that the stream function $\calG=\calG[\overline\omg]$ has the \textit{strict} Steiner symmetry in the half plane $\set{\bfx\in\bbR^2:x_2>0}$. We will establish several observations related to the concavity of $\calG$ on the $x_2$-axis. Knowing that the vortex core is simply connected and has the Steiner symmetry in the half plane $\set{x_2>0}$, we can find two constants $0\leq R_1<R_2<\ift$ such that 
$$\set{s>0:\overline\omg(0,s)>0}=(R_1,R_2).$$

\medskip
\noindent$\blacktriangleright$ \textit{(Claim 1) The function $s\mapsto\rd_{x_2}\calG(0,s)$ is strictly increasing
in the intervals $[0,R_1]$ and $[R_2,\ift)$, respectively.} 

\medskip
Considering the convolution formula \eqref{eq: stream}, we note that $\calG$ is infinitely differentiable at the point $\bfx\in\overline{\set{\overline\omg\neq0}}^c$ since the function $\log|\cdot|$ in the integrand is infinitely differentiable without any singularity. In particular, for each $\bfx\in\overline{\set{\overline\omg\neq0}}^c$, we have 
\begin{equation}\label{eq1117_2}
-\lap\calG(\bfx)=\overline\omg(\bfx)=0.
\end{equation}
For each $s\in(0,\ift)\setminus[R_1,R_2]$, the Steiner symmetry of $\calG$ gives $\rd_{x_1}^2\calG(0,s)\leq 0$. Indeed, we obtain the strict inequality $\rd_{x_1}^2\calG(0,s)<0$ by observing that, from \eqref{eq1117_1}, we get 
$$\begin{aligned}
2\pi\rd_{x_1}^2\calG(0,s)&=\int_{y_2>0}\rd_{x_1}\f{4x_2y_2(y_1-x_1)}{|\bfx-\bfy|^2|\bfx-\bfy^*|^2}\Bigg|_{\bfx=(0,s)}\overline\omg(\bfy)\dd\bfy\\
&=\int_{y_2>0}\f{4sy_2\left[3y_1^4+2y_1^2(s^2+y_2^2)-(s^2-y_2^2)^2\right]}{\left[(y_1^2+(s-y_2)^2)(y_1^2+(s+y_2)^2)\right]^2}\overline\omg(\bfy)\dd\bfy\\
&=\ii{R_1}{R_2}8sy_2\underbrace{\left(\ii{0}{\ift}\f{3t^4+2t^2(s^2+y_2^2)-(s^2-y_2^2)^2}{\left[(t^2+(s-y_2)^2)(t^2+(s+y_2)^2)\right]^2}\overline\omg(t,y_2)\dd t\right)}_{(*)}\dd y_2
\end{aligned}$$
where, since the numerator $3t^4+2t^2(s^2+y_2^2)-(s^2-y_2^2)^2$ changes its sign only at 
$$t^2=-\f{a}{3}+\sqrt{\f{a^2}{9}+\f{b^2}{3}}=:T>0,\qd a=s^2+y_2^2,\q b=s^2-y_2^2,$$ we have 
$$\begin{aligned}
(*) &=\ii{0}{\ift}(\cdots)\overline\omg(t,y_2)\dd t=\ii{0}{\sqrt{T}}(\cdots)\omg(t,y_2)\dd t+\ii{\sqrt{T}}{\ift}(\cdots)\overline\omg(t,y_2)\dd t\\
&\leq \q\overline\omg(\sqrt{T},y_2)\underbrace{\ii{0}{\sqrt{T}}(\cdots)\dd t}_{<0}\q+\q
\overline\omg(\sqrt{T},y_2)\underbrace{\ii{\sqrt{T}}{\ift}(\cdots)\dd t}_{>0} =\q\overline\omg(\sqrt{T},y_2)\ii{0}{\ift}
(\cdots)\dd t\\
&=\q\overline\omg(\sqrt{T},y_2)\lim_{t\to\ift}\f{t}{(t^2+(s-y_2)^2)(t^2+(s+y_2)^2)}=0
\end{aligned}$$ (note that we have $(*)=0$ only if $\overline\omg(\cdot,y_2)\equiv0$, which is impossible). Applying the strict inequality $\rd_{x_1}^2\calG(0,s)<0$ to the relation \eqref{eq1117_2} gives us  $\rd_{x_2}^2\calG(0,s)>0$. This completes the proof of \textit{(Claim 1)}.\\

\noindent$\blacktriangleright$ \textit{(Claim 2) The set $\set{s>0:\calG(0,s)-Ws>0}$ is nonempty and contains the interval $(0,R_2)$.}

\medskip
We recall that $\rd_{x_2}\calG(0,0)>2W$ from Proposition~\ref{prop: 2D central speed} and that $\rd_{x_2}\calG(0,s)$ does not decrease in $s\in[0,R_1]$ from \textit{(Claim 1)}, which imply  that $\calG(0,s)-Ws>Ws>0$ for each $s\in[0,R_1]$. Note that the set $\set{s>0:\calG(0,s)-Ws>0}$ is nonempty in either cases of $R_1=0$ or $R_1>0$. It remains to show $$R_2\leq R:= \sup\set{x_2>0: \calG(0,s)-Ws>0\q\q\mbox{for each}\q s\in(0,x_2)}.$$
We suppose the contrary, i.e., $R<R_2.$ By the continuity of $\calG,$ we have $\calG(0,R)-WR=0$ and so $R>R_1$. In sum, we have $R\in(R_1,R_2)=\set{s>0:\overline\omg(0,s)>0}$. Since the vortex core is open, there exists a small constant $0<\eps\ll1$ such that $\overline\omg(-\eps,R)>0$. Moreover, by the \textit{strict} Steiner symmetry of $\calG$, we have $\calG(-\eps,R)-WR=:\sgm<0$. Since $\calG(0,x_2)-Wx_2>0>\sgm$ for each $x_2\in(0,R)$ by the definition of $R$, Lemma~\ref{lem: 2D streamline formation} implies that the point $(-\eps,R)$ cannot be contained in the vortex domain (and hence in the vortex core),  which is a contradiction to $\overline\omg(-\eps,R)>0$. Thus, we conclude that $R\geq R_2,$ and this completes the proof of \textit{(Claim 2)}.\\

\noindent$\blacktriangleright$ \textit{(Claim 3) We have $\rd_{x_2}\calG(0,s)<0$ for $s\geq R_2$.}

\medskip
We observe that the derivative $\rd_{x_2}\calG(0,s)$ is strictly increasing in $s\in[R_2,\ift)$ by \textit{(Claim 1)}, while $|\rd_{x_2}\calG(0,s)|\to 0$ as $s\to\ift$ since $|\nb\calG(\bfx)|\leq \f{1}{2\pi}\int_{\bbR^2}\f{1}{|\bfx-\bfx'|}|\overline\omg(\bfx)|d\bfx'=O(|\bfx|^{-1})$ as $|\bfx|\to\ift$. This means $\rd_{x_2}\calG(0,s)<0$ for $s\geq R_2$. It completes the proof of \textit{(Claim 3)}.\\

\noindent$\blacktriangleright$ \textit{(Claim 4) The superlevel set $\set{\bfx\in\bbR^2:x_2>0,\,\calG(\bfx)-Wx_2>0}$ is simply connected and has the Steiner symmetry.}

\medskip
We recall \textit{(Claim 2)} and \textit{(Claim 3)}, which are saying that the function $\calG(0,s)-Ws$ is positive for each $s\in(0,R_2)$ and decreases strictly in $s>R_2$. By the boundedness of the superlevel set $\set{\calG(\bfx)-Wx_2>0}$ (Proposition~\ref{prop: decay of stream/term}), we have $\set{s>0:\calG(0,s)-Ws>0}=(0,R)$ for some $R\in(0,\ift)$. By the Steiner symmetry of $\calG$, the superlevel set $\set{\bfx\in\bbR^2:x_2>0,\,\calG(\bfx)-Wx_2>0}$ has the Steiner symmetry and is simply connected. It completes the proof of \textit{(Claim 4)}.\\

It remains to prove that the superlevel set $$\left\{\bfx\in\bbR^2:x_2>0,\,\calG(\bfx)-Wx_2>0\right\}$$ (which is simply connected and has the Steiner symmetry by \textit{(Claim 4)}) contains the vortex core $\set{\bfx\in\bbR^2:x_2>0,\,\overline\omg>0}$. We suppose the contrary, i.e., there exists $\bfy=(y_1,y_2)\in\bbR^2$ such that $y_1<0$, $y_2\in(R_1,R_2)$ satisfying  $$\overline\omg(\bfy)>0\qd\mbox{and}\qd\calG(\bfy)-Wy_2\leq0.$$ Since the vortex core is an open set and $\calG$ has the \textit{strict} Steiner symmetry, for small $0<\eps\ll1$ and $\bfy_\eps:=(y_1-\eps,y_2)\in\set{x_1<0,x_2>0}$ we have
$$\overline\omg(\bfy_\eps)>0\qd\mbox{and}\qd\calG(\bfy_\eps)-Wy_2=:\gmm<0.$$ 
Using the result of \textit{(Claim 2)} and Lemma~\ref{lem: 2D streamline formation},  we derive that the point $\bfy_\eps$ cannot be contained in the vortex domain, which is a contradiction to $\overline\omg(\bfy_\eps)>0$. It completes the proof of Lemma~\ref{lem: 2D core in superlevel set}.

\end{proof}

Here, we give the proof of Theorem~\ref{thm: class. of 2D atmos.}, which saying that the core-containing superlevel set $\set{\calG[\overline\omg]-Wx_2>0}$ in Lemma~\ref{lem: 2D core in superlevel set} above is indeed the vortex domain.

\begin{proof}[\textbf{Proof of Theorem~\ref{thm: class. of 2D atmos.}}]
Let $\Omg\subset\bbR^2$ be the vortex domain corresponding to a given dipole $\overline\omg.$ As in the proof of Lemma~\ref{lem: 2D core in superlevel set}, the assumption about the vortex core $\set{x_2>0,\overline\omg>0}$ implies the existence of two constants $0\leq R_1<R_2<\ift$ such that 
$$\set{s>0:\overline\omg(0,s)>0}=(R_1,R_2).$$ 
For the proof, we note that the stream function $\calG=\calG[\overline\omg]$ has the \textit{strict} Steiner symmetry (Lemma~\ref{lem: 2D streamline formation}) Also, we recall \textit{(Claim 2)}, \textit{(Claim 3)}, \textit{(Claim 4)}  in the proof of Lemma~\ref{lem: 2D core in superlevel set}:
$$\begin{cases}
\text{ \textit{(Claim 2)}:\q The set $\set{s>0:\calG(0,s)-Ws>0}$ is nonempty and contains the interval $(0,R_2)$.}\\
\text{ \textit{(Claim 3)}:\q We have $\rd_{x_2}\calG(0,s)<0$ for $s\geq R_2$.}\\
\text{ \textit{(Claim 4)}:\q The superlevel set $\set{\bfx\in\bbR^2:x_2>0,\,\calG(\bfx)-Wx_2>0}$ is simply connected and has the Steiner symmetry.}
\end{cases}$$ 

To prove Theorem~\ref{thm: class. of 2D atmos.}, it suffices to show that the vortex domain 
coincides with the superlevel set, i.e., 
\begin{equation}\label{eq1031_1}
\Omg\cap\set{x_2>0}=\set{\calG-Wx_2>0}\cap\set{x_2>0}.
\end{equation} Indeed, the superlevel set is simply connected and has the Steiner symmetry by \textit{(Claim 4)}, which also implies the existence of the boundary curve $l:(0,R]\to[0,\ift)$ in the statement of Theorem~\ref{thm: class. of 2D atmos.} as desired (the differentiability of $l$ is deduced from the differentiability and the \textit{strict} Steiner symmetry of $\calG$), together with the end point value $l(0):=\lim_{x_2\downarrow0} l(x_2)\in(0,\ift)$ (it is well defined by the monotonicity of $\rd_{x_2}\calG(s,0)$ in $s\geq0$; see \cite[pg. 54-55]{CJS2025} for details) and the constant $R\in[R_2,\ift)$ satisfying 
\begin{equation}\label{eq1106_3}
\calG(0,s)-Ws>0\qd\mbox{for}\q s\in(0,R)\qd\mbox{and}\qd\calG(0,R)-WR=0
\end{equation}
Indeed, such a constant $R\in[R_2,\ift)$ exists uniquely from the facts that the function $\calG(0,s)-Ws$ is positive if $0<s<R_2$ (\textit{Claim 2}), decreases strictly in $s\geq R_2$ (\textit{Claim 3}), and tends to $-\ift$ as $s\to\ift$ (Proposition~\ref{prop: decay of stream/term})).

\medskip
We will prove \eqref{eq1031_1} by a contradiction argument. Since we have seen that superlevel set $\set{\calG-Wx_2>0}$ is contained in the vortex domain (Section~\ref{sec: streamline}), we assume that there exists a point $\bfy\in\Omg\cap\set{x_2>0}$ satisfying $\calG(\bfy)-Wy_2\leq0$. Since the vortex domain $\Omg\cap\set{x_2>0}$ is an open set and $\calG$ has the \textit{strict} Steiner symmetry, there exists a small constant $0<|\eps|\ll1$ such that we have $y_1+\eps\neq0$, $(y_1+\eps,y_2)\in\Omg\cap\set{x_2>0}$, and  $$0\geq\calG(\bfy)-Wy_2>\calG(y_1+\eps,y_2)-Wy_2=:\gmm.$$ In other words, the vortex domain $\Omg\cap\set{x_2>0}$ contains at least one point $\bfy'=(y_1',y_2')=(y_1+\varepsilon,y_2)$ such that $\calG(\bfy')-Wy_2'=\gmm<0$ and $y_1'\neq0$. In what follows, either case $y_1>0$ or $y_1<0$ will face a contradiction.

\medskip
\noindent\textbf{Case \circled{1}: $y_1'<0$}

In this case, note that \textit{(Claim 2)}, \textit{(Claim 3)}, and Lemma~\ref{lem: 2D streamline formation} say that $\bfy'$ cannot be in the vortex domain $\Omg$ in either cases of $y_2'\in(0,R_2]$ and $y_2'>R_2$, so a  contradiction. More precisely, if $y_2'\in(0, R_2]$, then due to \textit{(Claim 2)}, we can directly use Lemma \ref{lem: 2D streamline formation} to derive a contradiction. So we may assume $y_2'>R_2$. Then, by the \textit{strict} Steiner symmetry of $\calG$, we know $\calG(0,y_2')-Wy_2'>\gamma$. From \textit{(Claim 3)}, we get
$\calG(0,s)-Ws>\gamma$ for every $s\in[R_2,y_2']$, and hence for every $s\in(0,y_2']$ by \textit{(Claim 2)}. Again, Lemma \ref{lem: 2D streamline formation} gives a contradiction, so \textbf{Case \circled{1}} is impossible.\\

\noindent\textbf{Case \circled{2}: $y_1'>0$}

In this case, the idea is that we can find the trajectory starting from $\bfy'$ that intrudes the region $\set{x_1<0}$ in a finite time. Then we can follow the same argument with \textbf{Case \circled{1}}. For completeness, we give the details below.

For each $x_2>0$, we consider the equation \begin{equation}\label{eq0921_1}
\calG(x_1,x_2)-Wx_2=\gmm<0 \qd\mbox{for}\q x_1\geq0.
\end{equation} Since $\calG>0$ in $\set{x_2>0}$, the equation is not solvable unless $x_2>-\gmm/W>0.$ Moreover, even if we assume $x_2>-\gmm/W$, we note that the equation is not solvable for any $x_2>R'$, where the constant $R'>R_2$ satisfies $\calG(0,R')-WR'=\gmm$ (such constant $R'$ exists uniquely by the same reason with the case of $R$; see \eqref{eq1106_3}). Moreover, \eqref{eq0921_1} is solvable for each $x_2\in(-\gmm/W,R')$, since $\calG(0,s)-Ws>\gmm$ for each $s\in(0,R')$ from \textit{(Claim 2)} and \textit{(Claim 3)}. In sum, the equation \eqref{eq0921_1} is solvable only when $x_2\in(-\gmm/W,R')$. 

By the \textit{strict} Steiner symmetry of $\calG$, there exists a function $f:(-\gmm/W,R']\to[0,\ift)$ satisfying $f>0$ in $(-\gmm/W,R')$, $f(R')=0$, and that solution of \eqref{eq0921_1} is uniquely determined as $x_1=f(x_2)$ for $x_1\in (-\gmm/W,R']$. The differentiability of $f$ in $(-\gmm/W,R']$ is easily obtained by the differentiability and the \textit{strict} Steiner symmetry of $\calG$. From this, we observe that \begin{equation}\label{eq0921_2}
\lim_{s\downarrow-\gmm/W}f(s)=\ift.
\end{equation} Otherwise, if $\liminf_{s\downarrow-\gmm/W}f(s)\in[0,\ift)$, then for some subsequence $s_n\to -\gmm/W$, we have
$$0=\lim_{n\to\ift}Ws_n+\gmm=\lim_{n\to\ift}\calG(-f(s_n),s_n)= \calG\left(-\liminf_{s\downarrow-\gmm/W}f(s),-\gmm/W\right)>0,$$ which is a contradiction. 

We now return to the point $\bfy'=(y_1',y_2')$ that clearly solves \eqref{eq0921_1}.
We have $-\overline{u}^2=\rd_{x_1}\calG<0$ on the right side $\set{x_1>0}$ (by \eqref{eq1117_1}), implying that the particle starting at $\bfy'$ increases its $x_2$-coordinate along the curve $\set{(f(s),s):s\in(-\gmm/W,R']}$ so that it reaches the point $(0,R')$ in a finite time $T>0$, where the horizontal velocity $\overline{u}^1(0,R')=\rd_{x_2}\calG(0,R')$ is negative by \textit{(Claim 3)} (recall $R'>R_2$). More precisely, from $\overline{u}^1(0, R')<0$, there exists $\delta>0$ satisfying
\begin{equation}\label{eq: temp_u_1}
\sup_{|\bfx-(0,R')|<\delta} \overline{u}^1(\bfx)<0. 
\end{equation}
From $f(s)\to0$ as $s\uparrow R'$, there exists $\dlt'\in(0,\dlt)$ such that $f(s)<\dlt/2$ whenever $s>R'-\dlt'$. Note
$\inf_{\bfx\in L}\overline u^2(\bfx)>0$ where 
the compact set $L$ is defined by $L:=\{
(f(s),s)\,:\, y_2'\leq s \leq R'-(\dlt'/2)
\}\subset \{x_1>0, \,x_2>0\}$. Thus the trajectory starting from $\bfy'$ arrives at some point $\bfx$ satisfying $|\bfx-(0,R')|<\delta$ with $x_1>0$ and $x_2<R'$  in finite time. Then, due to \eqref{eq: temp_u_1}, the particle crosses the $x_2$-axis at some finite time from the right $\set{x_1>0}$ to the left $\set{x_1<0}$, which reminds us \textbf{Case \circled{1}} that was already negated. It completes the proof of Theorem~\ref{thm: class. of 2D atmos.}.
\end{proof}

\subsection{Sadovskii vortex}\label{sec: Sadovskii}\q

\medskip Before we move toward the case of a 3D vortex ring, we consider a special form of a 2D vortex dipole: a \textit{Sadovskii vortex}. It is a vortex dipole satisfying the Steiner symmetry in the half plane $\set{x_2>0}$ and, in particular, exhibiting \say{touching} of the counter-rotating vortex pair in the sense of
$$\set{|\bfx|<r}\subset 
\overline{\set{\overline\omg\neq0}}\qd\mbox{for some}\q r>0.
$$ We refer to \cite{CJS2025, CSW2025} for backgrounds about this. Knowing that the vortex domain of every dipole has an oval shape in $\bbR^2$ (Theorem~\ref{thm: class. of 2D atmos.}), a natural direction to find a Sadovskii vortex has been to construct an odd-symmetric dipole $\overline\omg$ satisfying 
\begin{equation}\label{eq1125_1}
\set{\overline\omg>0}=\set{\calG[\overline\omg]-Wx_2>0}\qd\mbox{in the half plane}\q \set{x_2>0}
\end{equation} (see \cite{CJS2025, CSW2025, HT2025} and the references therein) so that the entire vortex core $\set{\overline\omg\neq0}$ coincides with the oval vortex domain in $\bbR^2$, i.e., the atmosphere is empty.

\medskip
In what follows, we will verify that the atmosphere of a Sadovskii vortex is always empty, and vice versa when the core is simply connected in $\mathbb{R}^2_+$\footnote{Typical examples are Chaplygin--Lamb dipole \cite{Chap1903, Lamb1993} and the Sadovskii patch found by Huang--Tong \cite{HT2025}.}. This would imply that characterizing a Sadovskii vortex by the relation \eqref{eq1125_1} is indeed the only possible direction.

\begin{thm}\label{thm: sadovskii}
Let $\overline\omg$ be a vortex dipole (Definition~\ref{def: 2d dipole}) has a simply-connected vortex core $\set{\overline\omg>0}$ in the half plane $\set{\bfx\in\bbR^2:x_2>0}$. Then, the corresponding vortex atmosphere is empty if and only if $\overline\omg$ is a Sadovskii vortex:
$$\set{\bfx\in\bbR^2: |\bfx|<r}\subset\overline{\set{\overline\omg\neq0}}\qd\mbox{for some}\q r>0.$$ 
\end{thm}
\begin{proof}

Let $\overline\omg$ be a vortex dipole having the traveling speed $W>0$. We first recall Theorem~\ref{thm: class. of 2D atmos.}, which saying that the vortex domain $\Omg$ has an oval shape in that, for a constant $R>0$ and a continuous $l:[0,R]\to[0,\ift)$ satisfying $l>0$ in $[0,R)$ and $l(R)=0$, we have
$$\Omg\cap\set{x_2>0}=\set{\bfx\in\bbR^2:x_2\in(0,R),\, |x_1|<l(x_2)}.$$ Hence, it is obvious that, assuming that the vortex core is simply connected in the half plane $\set{x_2>0}$, the empty atmosphere ((core)$\equiv$(vortex domain)) always implies that $\overline\omg$ is a Sadovskii vortex. 

\medskip
For the other direction, we now assume that $\overline\omg$ is a Sadovskii vortex. It suffices to show that the atmosphere is empty, i.e., the core coincides with the vortex domain $\Omg$. Knowing that the vortex core in the half plane $\set{x_2>0}$ is simply connected, has the Steiner symmetry, and contains a ball $\set{|\bfx|<r}$ for some $r>0$, we can find a constant $R'\geq r$ and a function $\tld{l}:[0,R')\to[0,\ift)$ satisfying $\tld{l}>0$ in $[0,R')$ and 
$$\set{\overline\omg>0}=\set{\bfx\in\bbR^2:x_2\in(0,R'),\, |x_1|<\tld{l}(x_2)}$$ (at this stage, the continuity of $\tld{l}$ and the end-point value $\tld l(R')$ are unknown, but later it will turn out to be true by showing that $R=R'$ and $l\equiv\tld l$). We also bring \textit{(Claim 2)} in the proof of Lemma~\ref{lem: 2D core in superlevel set} in here: the set $\set{s>0:\calG(0,s)-Ws>0}=(0,R)$ is nonempty and contains the interval $(0,R')$, where $\calG=\calG[\overline\omg]$ is the stream function. This implies $R\geq R'$.\\

\noindent$\blacktriangleright$\textit{(Claim 1) We have $l\equiv\tld l$ in the interval $(0,R')$.}

\medskip 
We clearly have $l\geq \tld{l}$, since the vortex domain must contain the vortex core. Suppose there exist $y_2\in (0,R')$ such that $l(y_2)>\tld{l}(y_2)$, and consider the point 
$$\bfy=\left(-\f{l(y_2)+\tld{l}(y_2)}{2}, y_2\right).$$ Observe that the vortex domain contains $\bfy$, but the closure of the core does not. So for all time $t>0$, the particle initially at $\bfy$ must stay in the vortex domain and, at the same time, not in the core. Since the point $\bfy$ is in the left side $\set{x_1<0}$ where the vertical velocity $\overline u^2=-\rd_{x_1}\calG$ is negative (due to the \textit{strict} Steiner symmetry of $\calG$), the particle initially at $\bfy$ must stay in the set 
$$\left\{\bfx\in\bbR^2:0<x_2\leq y_2,\, x_1\in (-l(x_2),-\tld l(x_2))\right\},$$ where the interval
$(-l(x_2),-\tld l(x_2))$ should be non-empty for each $0<x_2\leq y_2$.
Hence, the particle must be heading downward for all time without any standstill (see the proof of Lemma~\ref{lem: 2D core in superlevel set} for a similar argument). It implies that the particle goes down forever to meet $x_1$-axis at some point $(-L,0)$ 
at time infinity satisfying the relation $$0<\tld{l}(0)\leq L\leq l(0).$$ As every particle trajectory conserves the value of $\calG-Wx_2$, we obtain the relation
$$\calG(-L,0)-W\cdot0=\calG(\bfy)-Wy_2.$$ Here, a contradiction arises. The right-hand side is positive since $\bfy$ is contained in the vortex domain $\Omg\cap\set{x_2>0}$ given as the superlevel set $\set{\calG-Wx_2>0}\cap\set{x_2>0}$ (Theorem~\ref{thm: class. of 2D atmos.}), while the left-hand side is clearly $0$ due to the odd-symmetry of $\calG$ (see \eqref{eq1125_3} in the proof of Lemma~\ref{lem: 2D streamline formation}). It completes the proof of \textit{(Claim 1)}.\\

\noindent$\blacktriangleright$\textit{(Claim 2) We have $l(R')=0$.}

\medskip
We suppose the contrary, i.e., assume $l(R')>0$. From the result of \textit{(Claim 1)}, we have $R>R'$. Moreover, we may define $\tld l(R'):=l(R')$ and consider the extension of $\tld l$ such that $\tld l\equiv0$ in $(R',R)$, since the vortex core is contained in $\set{x_2<R'}$. Then, in the interval $(R,R')$, we have $l>0$ and $\tld l\equiv0$. Consider any point $\bfz$ in the nonempty set $$\calA:=\set{\bfx\in\bbR^2:x_2\in(R',R),\,x_1\in(-l(x_2),0)}.$$ Since $\bfz$ is in the vortex domain and not in the closure of the core, the particle initially at $\bfz$ must be staying in the set $\calA$. Here, the reason that the particle cannot cross the $x_2$-axis in a finite time is that the horizontal velocity is negative on the line segment $\set{x_1=0,x_2\in[R',R)}$:
\begin{equation}\label{eq1126_1}
\overline u^1(0,s)=\rd_{x_2}\calG(0,s)<0\qd\mbox{for each}\q s\in[R',R)
\end{equation}
(recall \textit{(Claim 3)} in the proof of Lemma~\ref{lem: 2D core in superlevel set}). Since $\calA$ is in the left side $\set{x_1<0}$ where the vertical velocity $\overline u^2=-\rd_{x_1}\calG$ is negative (due to the \textit{strict} Steiner symmetry of $\calG$), the particle must be heading downward for all time without any standstill, which means it needs to escape from $\calA$ by crossing the line segment $$\calL:=\set{x_1\in[-l(R'),0],\,x_2=R'}$$ in a finite time, so a contradiction. This completes the proof of \textit{(Claim 2)}.\\

From \textit{(Claim 1)} and \textit{(Claim 2)}, we conclude that $R=R'$ and $l\equiv\tld l$ in the interval $(0,R]$. We may redefine $\tld l(0)=l(0)$, if necessary. It means the vortex core $\set{\overline\omg>0}$ coincides with the vortex domain $\Omg\cap\set{x_2>0}$, and hence the atmosphere is empty. It completes the proof of Theorem~\ref{thm: sadovskii}. \end{proof}

\subsection{Ring with simply-connected core}\q

\medskip
 Unlike the 2D case (Theorem~\ref{thm: class. of 2D atmos.}), the 3D setting (vortex rings)  admits two additional types of vortex domains.
\begin{thm}\label{thm: class. of 3D atmos.}
Let $\overline\xi$ be any vortex ring (Definition~\ref{def: 3D ring}) with a traveling speed $W>0$. In the $(z,r)$-plane $\Pi$, if $\overline\xi$ has the Steiner symmetry (Definition~\ref{def: 3D steiner}) with a simply-connected vortex core, then the corresponding vortex domain $\Omg$ also has the Steiner symmetry and is simply connected in $\Pi$. Moreover, for the $z$-component $\overline{v}^z=\rd_r\psi[\overline\xi]/r$ of the velocity $\overline\bfv$, one of the following holds:

\medskip
{\rm\textbf{[Case I]}} \say{$\overline{v}^z(0,0)\geq W$}: The vortex domain is given by the superlevel set $$\Omg=\left\{(z,r)\in\Pi:
\psi[\overline\xi](z,r)-Wr^2/2>0\right\}.$$ Moreover, for some constant $R\in(0,\ift)$ and a function $l:[0,R]\to[0,\ift)$ satisfying $l>0$ in $(0,R)$, $l(R)=0$, and $l\in C([0,R])\cap C^1((0,R))$, the vortex domain is given as
$$\Omg=\left\{(z,r)\in\Pi:r\in(0,R),\,|z|<l(r)\right\}.$$ If $\overline{v}^z(0,0)>W$, we have $l(0)\in(0,\ift)$, implying the \say{spheroidal shape} of the above vortex domain in $\bbR^3$. If $\overline{v}^z(0,0)=W$, we have $l(0)=0$, implying the \say{revolved-lemniscate shape} of the above vortex domain in $\bbR^3$.

\medskip
{\rm\textbf{[Case II]}} \say{$\overline{v}^z(0,0)<W$}: The vortex domain $\Omg$ satisfies $\Omg\supsetneq\set{(z,r)\in\Pi:\psi[\overline\xi](z,r)-Wr^2/2>0}$ and 
is given by the superlevel set $$\Omg=\left\{(z,r)\in\Pi:\psi[\overline\xi](z,r)-Wr^2/2>\gmm,\q r>L\right\}$$ for some constants $\gmm<0$ and $L>0$. In particular, we have $\psi[\overline\xi](0,L)-WL^2/2=\gmm$, and the map
$r\mapsto\overline{v}^z(0,r)$ strictly increases in the interval $(0,L)$ to reach $$\overline{v}^z(0,L)=W.$$
Moreover, the vortex domain has a \say{toroidal shape} in the sense that, for some constant $R\in(L,\ift)$ and a function $l:[L,R]\to[0,\ift)$ satisfying $l>0$ in $(L,R)$, $l(L)=l(R)=0$, and $l\in C([L,R])\cap C^1((L,R))$, the vortex domain $\Omg$ is given by 
$$\Omg=\left\{(z,r)\in\Pi:r\in(L,R),\,|z|<l(r)\right\}.$$ See Figure~\ref{fig: 3D donut} for a simple illustration.
\end{thm}

\begin{figure}[t]
\centering
\includegraphics[width=0.7\linewidth]{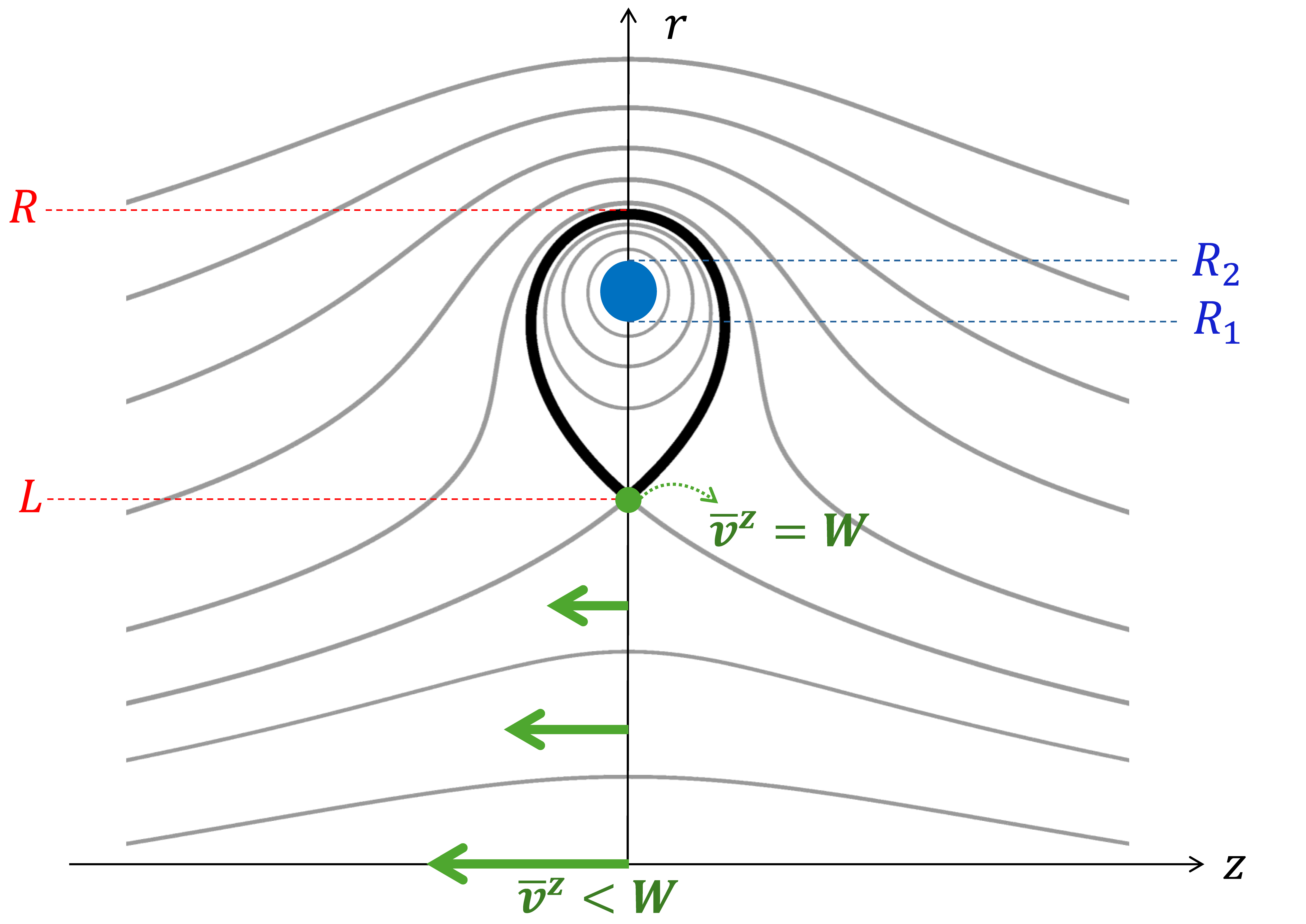}
\caption{The figure describes   \textbf{[Case II]}-Theorem~\ref{thm: class. of 3D atmos.} in the $(z,r)$-plane $\Pi$. The black solid line shows the boundary of the toroidal shape vortex domain $\Omg$, intersecting with the $r$-axis at two edge points $(0,L)$ and $(0,R)$. Each green arrow denotes the (relative) velocity $\overline\bfv-(W,0)$  on the line segment $\set{z=0,\, 0\leq r\leq L}$. Its $z$-component is negative at the origin and strictly increases its value along the $r$-axis to reach $0$ at the point $(0,L)$. The vortex core, which is filled with the blue color, is contained in $\Omg$, having two edge points $(0,R_1)$ and $(0,R_2)$ satisfying $L\leq R_1<R_2\leq R$. 
} 
\label{fig: 3D donut}
\end{figure}

The proof of Theorem~\ref{thm: class. of 3D atmos.} requires following two lemmas, Lemma~\ref{lem: 3D streamline formation} and Lemma~\ref{lem: 3D core in superlevel set}. This subsection will be completed with the proof of Theorem~\ref{thm: class. of 3D atmos.}. 

\begin{lem}\label{lem: 3D streamline formation}
Let $\overline\xi$ be any vortex ring (Definition~\ref{def: 3D ring}) with a traveling speed $W>0$. In the $(z,r)$-plane $\Pi$, if $\overline\xi$ has the Steiner symmetry (Definition~\ref{def: 3D steiner}), then its stream function $\psi[\overline\xi]$ has the \textit{strict} Steiner symmetry. Moreover, any point $(z',r')\in\set{(z,r)\in\Pi:z<0}$ satisfying $$\psi[\overline\xi](z',r')-Wr'^2/2=:\gmm'<0\qd\mbox{and}\qd 
\psi[\overline\xi](0,s)-Ws^2/2>\gmm'\qd\mbox{for each}\q 
s\in(\sqrt{-2\gmm'/W},r')
$$ is not contained in the corresponding vortex domain. See Figure~\ref{fig: 3D streamline formation} for a simple illustration for this.
\end{lem}

\begin{figure}[t]
\centering
\includegraphics[width=0.9\linewidth]{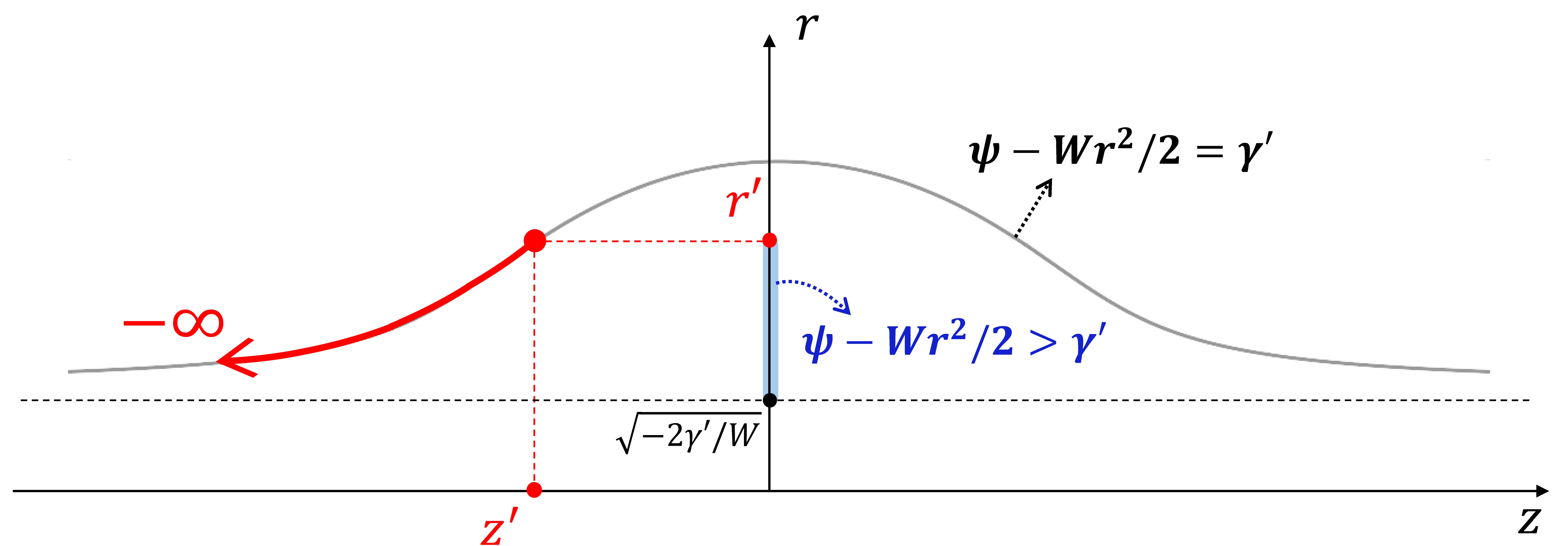}
\caption{In Lemma~\ref{lem: 3D streamline formation}, the flow map in the $(z,r)$-plane $\Pi$ takes the particle at $(z',r')$ indefinitely away from the origin. For all time, this particle is imprisoned in the level set $\set{\psi[\overline\xi]-Wr^2/2=\gmm'}
$ and asymptotically approaches the horizontal line $\set{r=\sqrt{-2\gmm'/W}}$.}
\label{fig: 3D streamline formation}
\end{figure}
\begin{proof}
We first note that the stream function $\psi=\psi[\overline\xi]$ is given by 
$$\psi(z,r)=
\int_\Pi G(r,z,r',z')\overline\xi(z',r')r'\dd r'\dd z'>0$$ with the nonnegative Green's function 
$$\begin{aligned}
G(r,z,r',z')&=\f{rr'}{2\pi}\int_0^\pi\f{\cos\vartheta}{[r^2+r'^2-2rr'\cos\vartheta+(z-z')^2]^{1/2}}\dd\vartheta\\
&=\f{rr'}{2\pi}\int_0^{\pi/2}\cos\vartheta\left(\f{1}{[r^2+r'^2-2rr'\cos\vartheta+(z-z')^2]^{1/2}}-\f{1}{[r^2+r'^2+2rr'\cos\vartheta+(z-z')^2]^{1/2}}\right)\dd\vartheta\\
&>0.
\end{aligned}$$ We recall \textbf{Step II} in the proof of Theorem \ref{thm: 3D donut}, which used the Steiner symmetry of $\xi_\lmb$ only, to confirm the \textit{strict} Steiner symmetry of $\psi$ with 
$\partial_z\psi<0$ in $\set{(z,r)\in\Pi:z>0}$.

\medskip Now, we suppose that there exists a point $(z',r')\in\set{(z,r)\in\Pi:z<0}$ satisfying 
$$\psi(z',r')-Wr'^2/2=:\gmm'<0\qd\mbox{and}\qd 
\psi(0,s)-Ws^2/2>\gmm'\qd\mbox{for each}\q s\in(\sqrt{-2\gmm/W},z').$$ 
We can repeat a similar argument with the proof of Lemma~\ref{lem: 2D streamline formation} thanks to the decay of $\psi(z,r)=O(|(z,r)|^{-1})$ given in \cite[Lemma 1.1]{FT1981} such that the point $(z',r')$ cannot be contained in the corresponding vortex domain. It completes the proof.
\end{proof}

\begin{lem}\label{lem: 3D core in superlevel set}
Let $\overline\xi$ be any vortex ring (Definition~\ref{def: 3D ring}) with a traveling speed $W>0$. In the $(z,r)$-plane $\Pi$, suppose $\overline\xi$ has the Steiner symmetry (Definition~\ref{def: 3D steiner}) with a simply-connected vortex core $\set{(z,r)\in\Pi:\overline\xi(z,r)>0}$ and assume
$$\overline{v}^z(0,0)\geq W,$$ where $\overline{v}^z=\rd_r\psi[\overline\xi]/r$ is the $z$-component of the velocity $\overline\bfv$. Then, the superlevel set 
$$\left\{(z,r)\in\Pi:\psi[\overline\xi](z,r)-Wr^2/2>0\right\}$$ also has the Steiner symmetry, is simply connected, and contains the vortex core.
\end{lem}
\begin{proof}
The proof is essentially the same as that of Lemma~\ref{lem: 2D core in superlevel set}. For completeness, we will give a proof in the way of reminding each step in the proof of Lemma~\ref{lem: 2D core in superlevel set}. Knowing that the vortex core $\set{\overline\xi>0}$ is simply connected and has the Steiner symmetry, we fix two constants $0\leq R_1<R_2<\ift$ such that 
$$\set{s>0:\overline\xi(0,s)>0}=(R_1,R_2).$$

\medskip
\noindent$\blacktriangleright$ \textit{(Claim 1) The function $s\mapsto \rd_r\psi(0,s)/s$ is strictly increasing in the intervals $[0,R_1]$ and $[R_2,\ift)$, respectively.}

\medskip
Note that $\psi$ is infinitely differentiable at each point $(z,r)\in\overline{\set{\overline\xi>0}}^c$ since the Green's function $G$ in the integrand does not possess any singularity. Applying the relation 
$$-\f{1}{r^2}\left(\rd_r^2\psi-\f{1}{r}\rd_r\psi+\rd_z^2\psi\right)=\overline\xi$$ (e.g., see \cite[Chapter 1]{FT1981} for a background) for each $(0,r)\in\overline{\set{\overline\xi>0}}^c$, we obtain that
\begin{equation}\label{eq1103_1}
\left(\rd_r^2\psi-\f{1}{r}\rd_r\psi\right)(0,r)>0
\end{equation}
due to $\overline\xi(0,r)=0$ and $\rd_z^2\psi(0,r)<0$.
Indeed, the Steiner symmetry of $\psi$ (in Lemma~\ref{lem: 3D streamline formation}) implies $\rd_z^2\psi(0,r)\leq0$, and the equality does not hold by the following observation: 
$$\begin{aligned}
\rd_z^2G\big|_{z=0}&=\f{rr'}{2\pi}\int_0^\pi\rd_z^2\f{\cos\vtht}{[r^2+r'^2-2rr'\cos\vtht+(z-z')^2]^{1/2}}\Bigg|_{z=0}\dd\vtht\\
&=\f{rr'}{2\pi}\int_0^\pi\cos\vtht\f{[2z'^2-(r^2+r'^2-2rr'\cos\vtht)]}{[r^2+r'^2-2rr'\cos\vtht+z'^2]^{5/2}}\dd\vtht\\
&=\f{rr'}{2\pi}\int_0^{\pi/2}\cos\vtht
\left(\f{[2z'^2-(r^2+r'^2-2rr'\cos\vtht)]}{[r^2+r'^2-2rr'\cos\vtht+z'^2]^{5/2}}-\f{[2z'^2-(r^2+r'^2+2rr'\cos\vtht)]}{[r^2+r'^2+2rr'\cos\vtht+z'^2]^{5/2}}\right)\dd\vtht\\
&=\f{rr'}{2\pi}\int_0^{\pi/2}\cos\vtht\left(\int_{r^2+r'^2-2rr'\cos\vtht}^{r^2+r'^2+2rr'\cos\vtht}4t\f{4z'^2-t^2}{(z'^2+t^2)^{7/2}}\dd t\right)\dd\vtht
\end{aligned}$$ leads to  
$$\begin{aligned}
\rd_z^2\psi(0,r)&=\int_\Pi \rd_z^2G(r,0,r',z')\overline\xi(z',r')r'\dd r'\dd z'\\&
=\int_0^{\pi/2}\cos\vtht\int_{R_1}^{R_2} \f{rr'^2}{\pi}\underbrace{\left(\int_0^\ift \overline\xi(z',r')\cdot\int_{r^2+r'^2-2rr'\cos\vtht}^{r^2+r'^2+2rr'\cos\vtht}4t\f{4z'^2-t^2}{(z'^2+t^2)^{7/2}}\dd t \dd z' \right)}_{(*)}\dd r'\dd\vtht
\end{aligned}$$ where, for the integrand 
$$\begin{aligned}
(*)=\int_{r^2+r'^2-2rr'\cos\vtht}^{r^2+r'^2+2rr'\cos\vtht}4t\left(\int_0^\ift \overline\xi(z',r')\f{4z'^2-t^2}{(z'^2+t^2)^{7/2}}\dd z'\right)\dd t=\int_{r^2+r'^2-2rr'\cos\vtht}^{r^2+r'^2+2rr'\cos\vtht}4t\left(\int_0^{t/2}\cdots\dd z'+\int_{t/2}^\ift\cdots\dd z'\right)\dd t, 
\end{aligned}$$ the Steiner symmetry of $\overline\xi$ is used to confirm the inequalities 
$$\int_0^{t/2}\cdots\dd z'\leq \overline\xi(t/2,r')\int_0^{t/2}\f{4z'^2-t^2}{(z'^2+t^2)^{7/2}}\dd z',\qd\int_{t/2}^\ift\cdots\dd z'\leq \overline\xi(t/2,r')\int_{t/2}^\ift\f{4z'^2-t^2}{(z'^2+t^2)^{7/2}}\dd z',$$ and hence
$$\int_0^\ift\cdots\dd z'<
\overline\xi(t/2,r')\int_0^\ift\f{4z'^2-t^2}{(z'^2+t^2)^{7/2}}\dd z'=0$$ using the relation 
$\int_0^\ift \f{4z^2-1}{(z^2+1)^{7/2}}\dd z=0$. The strict inequality $<$ above is induced from the fact that the equality never holds unless $\overline\xi(\cdot,r')\equiv0$, since the vortex core $\set{\overline\xi>0}$ is a bounded set. Hence we have $(*)<0$, and so $\rd_z^2\psi(0,r)<0$. It justifies the strict inequality \eqref{eq1103_1}, which implies that $\rd_r\psi(0,r)/r$ is (strictly) increasing for $r\in[0,R_1]$ and $r\in[R_2,\ift)$ respectively. It completes the proof of \textit{(Claim 1)}.\\

\noindent$\blacktriangleright$ \textit{(Claim 2) The set $\set{s>0:\psi(0,s)-Ws^2/2>0}$ is nonempty and contains the interval $(0,R_2)$.}

\medskip
We will first show that the set $A:=\set{r>0:\psi(0,s)-Ws^2/2>0\q\q\mbox{for each}\q s\in(0,r)}$ is nonempty. If $R_1>0$, from the initial assumption $$\overline{v}^z(0,0)=\lim_{s\downarrow0}\f{1}{s}\rd_r\psi(0,s)\geq W$$ and the fact from \textit{(Claim 1)} that $\rd_r\psi(0,s)/s$ is increasing for $s\in[0,R_1]$, we have 
$$\psi(0,s)-Ws^2/2=\int_0^s s'\left(\f{1}{s'}\rd_r\psi(0,s')-W\right)\dd s' >0
\qd\mbox{for each}\q s\in(0,R_1].$$ Hence, the set $A$ is nonempty. Even when $R_1=0$, if $\overline{v}(0,0)>W$, we can deduce $A\neq\emptyset$ from the continuity of $\psi$. We now suppose $R_1=0$ and $\overline{v}(0,0)=W$. We will show $A\neq\emptyset$ by contradiction. If $A=\emptyset$, we have $$L:=\inf_{s\in[0,R_2/2]}\left(\psi(0,s)-Ws^2/2\right)\in(-\ift,0]\qd\mbox{and}\qd\psi(0,r)-Wr^2/2=L\qd\mbox{for some}\q r\in(0,R_2/2].$$ Note that $r\in(R_1,R_2)$ and $\psi(0,s)-Ws^2/2\geq L$ for each $s\in(0,r)$. We now choose $z<0$ such that $(z,r)\in\set{\overline\xi>0}$ (it is possible since the core $\set{\overline\xi>0}$ is open) to observe that 
$$\psi(z,r)-\f{W}{2}r^2<L\leq\psi(0,s)-\f{W}{2}s^2\qd\mbox{for each}\q s\in(0,r),$$ where we used the \textit{strict} Steiner symmetry of $\psi$. Applying Lemma~\ref{lem: 3D streamline formation} to the above relation gives a contradiction to the choice $(z,r)\in\set{\overline\xi>0}$. Hence, we conclude that $A\neq\emptyset$.

It remains to show $R_2\leq R:=\sup A$. We suppose the contrary, i.e., $R<R_2$. By the continuity of $\psi$, we have $\psi(0,R)-WR^2/2=0$ and so $R>R_1$. In sum, we have $R\in (R_1,R_2)=\set{s>0:\overline\xi(0,s)>0}$. Since the vortex core is open in $\Pi$, there exists a small constant $0<\eps\ll1$ such that $\overline\xi(-\eps,R)>0$. Moreover, by the \textit{strict} Steiner symmetry of $\psi$, we have $\psi(-\eps,R)-WR^2/2=:\sgm<0$. Since $\psi(0,s)-Ws^2/2>0>\sgm$ for each $x_2\in(0,R)$ by the definition of $R$, Lemma~\ref{lem: 3D streamline formation} implies that the point $(-\eps,R)$ cannot be contained in the vortex domain,  which is a contradiction to $\overline\xi(-\eps,R)>0$. Thus, we conclude that $R\geq R_2,$ and this completes the proof of \textit{(Claim 2)}.\\

\noindent$\blacktriangleright$ \textit{(Claim 3) We have $\rd_r\psi(0,s)/s<0$ for $s\geq R_2$.}

\medskip
We observe that $\rd_r\psi(0,s)/s$ is strictly increasing in $s\in[R_2,\ift)$ by \textit{(Claim 1)}, while $|\rd_r\psi(0,s)/s|\to0$ as $s\to\ift$ (see \cite[Lemma 1.1]{FT1981} for the decay $|\nb\psi/r|=O(|(z,r)|^{-3})$ as $|(z,r)|\to\ift$). This means $\rd_r\psi(0,s)/s<0$ for $s\geq R_2$. It completes the proof of \textit{(Claim 3)}. \\

\noindent$\blacktriangleright$ \textit{(Claim 4) The superlevel set $\set{(z,r)\in\Pi:\psi(z,r)-Wr^2/2>0}$ is simply connected and has the Steiner symmetry.}

\medskip
Recalling \textit{(Claim 2)} and \textit{(Claim 3)}, the function $\psi(0,s)-Ws^2/2$ is positive for $s\in(0,R_2)$ and decreases for $s>R_2$ since its derivative $\rd_r\psi(0,s)-Ws=s(\rd_r\psi(0,s)/s-W)$ is negative for $s>R_2$. By the boundedness of the superlevel set $\set{(z,r)\in\Pi:\psi(z,r)-Wr^2/2>0}$ (Proposition~\ref{prop: decay of stream/term}), we have $\set{s>0:\psi(0,s)-Ws^2/2>0}=(0,R)$ for some $R\in(0,\ift)$. By the Steiner symmetry of $\psi$, the superlevel set  $\set{(z,r)\in\Pi:\psi(z,r)-Wr^2/2>0}$ has the Steiner symmetry and is simply connected. It completes the proof of \textit{(Claim 4)}.\\

It remains to prove that the superlevel set 
$$\set{(z,r)\in\Pi:\psi(z,r)-Wr^2/2>0}$$ (which is simply connected and has the Steiner symmetry by \textit{(Claim 4)}) contains the vortex core $\set{(z,r)\in\Pi:\overline\xi(z,r)>0}$. We suppose the contrary, i.e., for some $(z',r')\in\Pi$ such that $z'<0$ and $r'\in(R_1,R_2)$ satisfying 
$$\overline\xi(z',r')>0\qd\mbox{and}\qd\psi(z',r')-Wr'^2/2\leq0.$$ Since the vortex core is open in $\Pi$ and $\psi$ has the \textit{strict} Steiner symmetry, for small $0<\eps\ll1$ we have 
$$\overline\xi(z'-\eps,r')>0\qd\mbox{and}\qd\psi(z'-\eps,r')-Wr'^2/2=:\gmm<0.$$ Using the result of \textit{(Claim 2)} and Lemma~\ref{lem: 3D streamline formation}, we derive that the point $(z'-\eps,r')$ cannot be contained in the vortex domain, which is a contradiction to $\overline\xi(z'-\eps,r')>0$. It completes the proof of Lemma~\ref{lem: 3D core in superlevel set}.
\end{proof}

We now give the proof of Theorem~\ref{thm: class. of 3D atmos.}. 
We first treat [Case I], and the proof of [Case II] will follow.
\begin{proof}[\textbf{Proof of [Case I]-Theorem~\ref{thm: class. of 3D atmos.}}]
The proof is essentially the same as that of Theorem \ref{thm: class. of 2D atmos.}. We first put 
$$\set{s>0:\overline\xi(0,s)>0}=(R_1,R_2)\qd\mbox{for}\q 0\leq R_1<R_2<\ift.$$ The stream function $\psi=\psi[\overline\xi]$ has the \textit{strict} Steiner symmetry (Lemma~\ref{lem: 3D streamline formation}), and the superlevel set 
$\set{\psi-Wr^2/2>0}$ is simply connected and has the Steiner symmetry by \textit{(Claim 4)} in the proof of Lemma~\ref{lem: 3D core in superlevel set}, which implies the existence of the boundary curve $l:(0,R]\to[0,\ift)$ in the statement [Case I] of Theorem~\ref{thm: class. of 3D atmos.} as desired (whose differentiability is obvious from the differentiability and the \textit{strict} Steiner symmetry of $\psi$) except the value of $l(0)$, together with the constant $R\in[R_2,\ift)$ satisfying 
$$\psi(0,s)-Ws^2/2>0\qd\mbox{for}\q s\in(0,R)\qd\mbox{and}\qd \psi(0,R)-WR^2/2=0.$$ Indeed, such a constant $R$ exists uniquely from the facts that the function $\psi(0,s)-Ws^2/2$ is positive if $0<s<R_2$ by \textit{(Claim 2)} of Lemma~\ref{lem: 3D core in superlevel set}, decreases strictly in $s\geq R_2$ by \textit{(Claim 3)} of Lemma~\ref{lem: 3D core in superlevel set}, and tends to $-\ift$ as $s\to\ift$ (Proposition~\ref{prop: decay of stream/term}).  For $l(0)$, by considering the limit $r\downarrow0$ in the relation 
$\psi(l(r),r)-Wr^2/2=0,$ the observation $$\begin{aligned}
&\left|\f{1}{r^2/2}
\psi(z,r)-\overline{v}^z(z',0)\right|=\left|\f{1}{r^2/2}\int_0^r\rd_r\psi(z,s)\dd s-\overline{v}^z(z',0)\right| = \left| \f{1}{r^2/2}\int_0^r s\overline{v}^z(z,s)\dd s-\overline{v}^z(z',0)\right|\\&\qquad \leq
\f{1}{r^2/2}\int_0^r s|\overline{v}^z(z,s)-\overline{v}^z(z',0)|\dd s \leq \sup_{s\in(0,r)}|\overline{v}^z(z,s)-\overline{v}^z(z',0)| \to0\q\mbox{as}\q |(z,r)-(z',0)|\to0
\end{aligned}
$$ gives the relation $\overline{v}^z(l(0),0)=W$ where $l(0):=\lim_{s\downarrow0}l(s)$. Indeed, this limit is justified by observing that
\begin{equation}\label{eq1117_3}
\lim_{s\downarrow0}l(s)\q
\begin{cases}
\in(0,\ift)&\text{if}\q \overline{v}^z(0,0)>W,\\
=0&\text{if}\q\overline{v}^z(0,0)=W,
\end{cases}
\end{equation} which is a result of the \say{strict} monotonicity of $\overline{v}^z(z,0)$ for $z\geq0$ that will be shown below:
$$\begin{aligned}\overline{v}^z(z,0)&=\lim_{r\to0}\f{1}{r^2/2}\psi(z,r) =\lim_{r\to0}\f{1}{r^2/2}\int_\Pi G(r,z,r',z')\overline\xi(z',r')r'\dd r'\dd z'\\&=\lim_{r\to0}\f{1}{\pi}
\int_{\mathbb{R}}\int_{R_1}^{R_2}
\underbrace{\left(\int_0^\pi\f{\cos\vartheta}{r[r^2+r'^2-2rr'\cos\vartheta+(z-z')^2]^{1/2}}\dd\vtht\right)}_{(**)}\overline\xi(z',r')r'^2\dd r'\dd z'\end{aligned}$$
where 
$$\begin{aligned}
(**)&=\int_0^{\pi/2}\f{\cos\vartheta}{r}\left(\f{1}{[r^2+r'^2-2rr'\cos\vartheta+(z-z')^2]^{1/2}}-\f{1}{[r^2+r'^2+2rr'\cos\vartheta+(z-z')^2]^{1/2}}\right)\dd\vtht\\
&=:\int_0^{\pi/2}\f{\cos\vartheta}{r}\left(\f{1}{A^{1/2}}-\f{1}{B^{1/2}}\right)\dd\vtht =\int_0^{\pi/2}\f{\cos\vartheta}{r}\left(\f{B-A}{A^{1/2}B^{1/2}(A^{1/2}+B^{1/2})}\right)\dd\vtht\\
&=\int_0^{\pi/2}\f{4r'(\cos\vartheta)^2}{A^{1/2}B^{1/2}(A^{1/2}+B^{1/2})}\dd\vtht \to \int_0^{\pi/2}\f{2r'(\cos\vartheta)^2}{[r'^2+(z-z')^2]^{3/2}}\dd\vtht\qd\mbox{as}\q r\to0, 
\end{aligned}$$ and thus we get
$$\begin{aligned}
\overline{v}^z(z,0)&=\f{1}{\pi}\int_\Pi 
\int_0^{\pi/2}\f{2r'(\cos\vartheta)^2}{[r'^2+(z-z')^2]^{3/2}}\dd\vtht
\overline\xi(z',r')r'^2\dd r'\dd z'\\
&=\f{1}{\pi}\int_0^{\pi/2}2(\cos\vartheta)^2\int_0^\ift r'^3\underbrace{\left(\int_{-\ift}^\ift\f{\overline\xi(z',r')}{[r'^2+(z-z')^2]^{3/2}}\dd z'\right)}_{(***)}\dd r'\dd\vtht.
\end{aligned}$$
We deduce that $(***)$ attains its maximum at $z=0$ from the observation that 
$$\begin{aligned}\rd_z(***)&=\int_{-\ift}^\ift\f{-3(z-z')}{[r'^2+(z-z')^2]^{5/2}}\overline\xi(z',r')\dd z'\\
&=\int_0^{\ift}\f{3z'}{[r'^2+z'^2]^{5/2}}\underbrace{\left(\overline\xi(z'+z,r')-\overline\xi(-z'+z,r')\right)}_{\leq0}\dd z'<0\qd\mbox{for}\q z>0,
\end{aligned}$$  where the strict equality $<$ is deduced obviously from the fact that the relation $$\text{\say{$\q\overline\xi(z'+z,r')-\overline\xi(-z'+z,r')=0\qd\mbox{for almost every}\q z'\in(0,\ift)\q$}}$$ never holds for each $r'\in(R_1,R_2)$. It implies the \textit{strict} monotonicity of $\overline{v}^z(z,0)$ for $z\geq0$, which justifies \eqref{eq1117_3}.

\medskip
It remains to prove that the vortex domain coincides with the superlevel set, i.e., 
$$\Omg=\set{\psi-Wr^2/2>0}.$$ Its proof follows directly by the arguments after \eqref{eq1031_1} in the proof of Theorem~\ref{thm: class. of 2D atmos.}, using Lemma~\ref{lem: 3D streamline formation} and Lemma~\ref{lem: 3D core in superlevel set} instead of Lemma~\ref{lem: 2D streamline formation} and Lemma~\ref{lem: 2D core in superlevel set}, respectively. It completes the proof of [Case I] in Theorem~\ref{thm: class. of 3D atmos.}.
\end{proof}
\medskip
\begin{proof}[\textbf{Proof of [Case II]-Theorem~\ref{thm: class. of 3D atmos.}}]
We now consider the case $\overline{v}^z(0,0)<W$. We put 
$$\set{s>0:\overline\xi(0,s)>0}=(R_1,R_2)\qd\mbox{for}\q 0\leq R_1<R_2<\ift.$$ Before we start the proof, for later uses, we bring some properties of the velocity $\overline\bfv=(\overline{v}^z,\overline{v}^r)=\f{1}{r}(\rd_r\psi[\overline\xi],-\rd_z\psi[\overline\xi])$ from the proofs of Theorem~\ref{thm: 3D donut} and Lemma~\ref{lem: 3D core in superlevel set}. From now on, we denote $\psi=\psi[\overline\xi]$ as usual.\\

\textbf{[P1]:} On the $r$-axis, the horizontal velocity $\overline{v}^z(0,s)=\rd_r\psi(0,s)/s$ is strictly increasing in the intervals $s\in[0,R_1]$ and $s\in[R_2,\ift)$, respectively.\\

\textbf{[P2]:} On the $r$-axis, the horizontal velocity $\overline{v}^z(0,s)=\rd_r\psi(0,s)/s$ is negative in the interval $s\in[R_2,\ift)$.\\

\textbf{[P3]:} In the set $\set{(z,r)\in\Pi:z>0}$, the vertical velocity $\overline{v}^r=-\rd_z\psi/r$ is positive.\\ 

Recall that \textbf{[P1]} and \textbf{[P2]} are exactly the sames with \textit{(Claim 1)} and \textit{(Claim 3)} in the proof of Lemma~\ref{lem: 3D core in superlevel set}. We emphasize that, even though the assumption on the center speed $\overline{v}^z(0,0)$ of the current theorem is different from that of Lemma~\ref{lem: 3D core in superlevel set}, the proof of \textbf{[P1]} and \textbf{[P2]} does not rely on that assumption. \textbf{[P3]} can be derived directly by replacing $\xi_\lmb$ with $\overline\xi$ in \textbf{Step II} in the proof of Theorem \ref{thm: 3D donut}), which only used the Steiner symmetry of $\xi_\lmb$.

\medskip
We begin our proof by establishing the following claim.\\

\noindent$\blacktriangleright$ \textit{(Claim 1) The set $\set{r>0:\rd_r\psi(0,s)/s<W\q\q\mbox{for each}\q s\in(0,r)}$ is nonempty and given as an interval $(0,L]$ for some $0<L<\ift$ satisfying $\rd_r\psi(0,L)/L=W$. Moreover, we have $L\leq R_1$, and the function $\psi(0,r)-Wr^2/2$ is negative and decreases strictly for $0<r<L$.}

\medskip
It is obvious that the set $\set{r>0:\rd_r\psi(0,s)/s<W\q\q\mbox{for each}\q s\in(0,r)}$ is nonempty from the relation $v^z(0,s)=\rd_r\psi(0,s)/s$ for $s>0$. To find $L\in(0,\ift)$ satisfying $L\leq R_1$, we suppose the contrary, i.e., $\rd_r\psi(0,r)/r<W$ for each $0<r<L'$ for some $L'>R_1$. Then the function $\psi(0,r)-Wr^2/2$ is negative and decreases strictly for $0<r<L'$ since $\lim_{r\downarrow0}\psi(0,r)=0$ and
$\rd_r\left(\psi(0,r)-Wr^2/2\right)=r\left(\rd_r\psi(0,r)/r-W\right)<0.$ Using this, we observe that any point $(z',r')$ in the core $\set{\overline\xi>0}\cap\set{R_1<r<L'}$ such that $z'<0$ satisfies $$
\psi(z',r')-\f{W}{2}r'^2<\psi(0,r')-\f{W}{2}r'^2<\psi(0,r)-\f{W}{2}r^2<0\qd\mbox{for each}\q r\in(0,r')$$ (first inequality is due to the \textit{strict} Steiner symmetry of $\psi$ from Lemma~\ref{lem: 3D streamline formation}), implying that the point $(z',r')$ cannot be contained in the vortex domain by Lemma~\ref{lem: 3D streamline formation}, so a contradiction. It completes the proof of \textit{(Claim 1)}.\\

The above claim implies that $R_1>0$ and, indeed, $R_1\geq L$ for the constant $L>0$ satisfying $\rd_r\psi(0,s)/s<W$ for each $s\in(0,L)$ and $\rd_r\psi(0,L)/L=W$.  It confirms the following claim.\\

\noindent$\blacktriangleright$ \textit{(Claim 2) We take the constant
$$\gmm:=\psi(0,L)-WL^2/2.$$ Then, we have $\gmm<0$ and  $\psi(0,s)-Ws^2/2\in(\gmm,0)$ for each $s\in(0,L)$.}

\medskip
By \textit{(Claim 1)}, it is obvious that the function $s\mapsto\psi(0,s)-Ws^2/2$ is negative and decreasing strictly in the interval $(0,L)$, so we get $\psi(0,s)-Ws^2/2>\gmm$ for each $s\in(0,L)$. It proves \textit{(Claim 2)}.\\

\noindent$\blacktriangleright$ \textit{(Claim 3) The set $\set{r>L:\psi(0,r)-Wr^2/2>\gmm}$ is nonempty and contains the interval $(L,R_2)$.}

\medskip
We denote $$A:=\set{r>L:\psi(0,s)-Ws^2/2>\gmm\q\q\mbox{for each}\q s\in (L,r)}.$$  If $R_1>L$, then the relation $\rd_r\psi(0,L)/L=W$ and \textbf{[P1]} guarantee $\rd_r\psi(0,r)/r>W$ for each $r\in(L,R_1)$, and hence we get
\begin{equation}\label{eq1106_1}
\psi(0,r)-\f{W}{2}r^2=\gmm+\int_L^r s\left(\f{1}{s}\rd_r\psi(0,s)-W\right)\dd s>\gmm \qd\mbox{for each}\q r\in(L,R_1].
\end{equation}
That is, the set $A$ is nonempty. We now consider the case $R_1=L$, and assume that $A=\emptyset$. Then for any $R_2'\in(R_1,R_2)$, we have $$\gmm':=\inf_{s\in[L,R_2']}\left(\psi(0,s)-Ws^2/2\right)\in(-\ift,\gmm]\qd\mbox{and}\qd\psi(0,r)-Wr^2/2=\gmm'\qd\mbox{for some}\q r\in(L,R_2']$$ 
(such $r$ exists in the interval $(L,R'_2]$ from our assumption $A=\emptyset$). Note that $r\in(R_1,R_2)$ and 
\begin{equation}\label{eq1106_2}
\psi(0,s)-Ws^2/2\geq \gmm'\qd\mbox{for each}\q s\in(L,r).\end{equation} Moreover, \textit{(Claim 2)} says that \eqref{eq1106_2} also holds for each $s\in(0,L]$. Knowing this, we now choose any $z<0$ satisfying $(z,r)\in\set{\overline\xi>0}$ to observe that 
$$\psi(z,r)-\f{W}{2}r^2<\gmm'\leq\psi(0,s)-\f{W}{2}s^2\qd\mbox{for each}\q s\in(0,r),$$ where we used the \textit{strict} Steiner symmetry of $\psi$ for the first inequality. Applying Lemma~\ref{lem: 3D streamline formation} to the above relation gives a contradiction to the choice $(z,r)\in\set{\overline\xi>0}$. Hence, we conclude that $A\neq\emptyset$.

It remains to show $R_2\leq a:=\sup A$. We suppose the contrary, i.e., $a<R_2$. From the definition of $a$, we have $a>L$, and by the continuity of $\psi$, we have $\psi(0,a)-Wa^2/2=\gmm$ and so $a>R_1$ (in case of $R_1=L$, it is obvious that $a>R_1=L$, and for the case $R_1>L$, we can recall the relation \eqref{eq1106_1} that negates the possibility of $a\leq R_1$). In sum, we have $a\in (R_1,R_2)$ and so $\overline\xi(0,a)>0$. Since the vortex core is open in $\Pi$, there exists a small constant $0<\eps\ll1$ such that $\overline\xi(-\eps,a)>0$. Moreover, by the \textit{strict} Steiner symmetry of $\psi$, we have 
$$\psi(-\eps,a)-\f{W}{2}R^2<\gmm\leq \psi(0,s)-\f{W}{2}s^2\qd\mbox{for each}\q s\in(0,a)$$ (the second inequality with the case $s\in(0,L]$ follows from \textit{(Claim 2)}, and the case $s\in(L,a)$ is obvious by the definition of $a$). Applying Lemma~\ref{lem: 3D streamline formation} to the above relation gives a contradiction to $\overline\xi(-\eps,a)>0$. It completes the proof of \textit{(Claim 3)}.\\

\noindent$\blacktriangleright$ \textit{(Claim 4) The superlevel set $\set{\psi-Wr^2/2>\gmm}\cap\set{r>L}$ contains the vortex core $\set{\overline\xi>0}$.}

\medskip
By $L\leq R_1$ from \textit{(Claim 1)}, it is obvious that the vortex core is contained in $\set{r>L}$. Knowing this, we suppose the contrary, i.e., for some $(z,r)\in\set{\overline\xi>0}$, we have
$$\psi(z,r)-Wr^2/2\leq\gmm.$$ Without loss of generosity, we may assume $z\leq0$ (by the Steiner symmetry of $\overline\xi$ and $\psi$). Since the vortex core is open in $\Pi$ and $\psi$ has the \textit{strict} Steiner symmetry, for small $0<\eps\ll1$ we have $\overline\xi(z-\eps,r)>0$ and 
$$\psi(z-\eps,r)-\f{W}{2}r^2<\gmm\leq \psi(0,s)-\f{W}{2}r^2\qd\mbox{for each}\q s\in(0,r),$$ where the last inequality follows directly from \textit{(Claim 2)} and \textit{(Claim 3)}. Applying Lemma~\ref{lem: 3D streamline formation} to the above relation, we derive that the point $(z-\eps,r)$ cannot be contained in the vortex domain, which contradicts $\overline\xi(z-\eps,r)>0$. It completes the proof of \textit{(Claim 4)}.\\

\noindent$\blacktriangleright$ \textit{(Claim 5) The superlevel set $\set{\psi-Wr^2/2>\gmm}\cap\set{r>L}$ is bounded, simply connected, and has the Steiner symmetry. In particular, the set $\set{r>L:\psi(0,r)-Wr^2/2>\gmm}$ is given by an interval $(L,R)$ for some constant $R\geq R_2$.}

\medskip
By \textit{(Claim 3)}, we have 
$$\psi(0,r)-\f{W}{2}r^2>\gmm\qd\mbox{for each}\q r\in(L,R_2).$$ We also note that the function $r\mapsto\psi(0,r)-Wr^2/2$ is strictly decreasing in the interval $[R_2,\ift)$ by \textbf{[P2]}. By the boundedness of the superlevel set $\set{\psi-Wr^2/2>\gmm}\cap\set{r>L}$ (note that it is contained in the set $\set{\psi/r^2>W/2+\gmm/L^2}$, which is bounded due to $W/2+\gmm/L^2=\psi(0,L)/L^2>0$ and the decay of $\psi/r^2$ given in Proposition~\ref{prop: decay of stream/term}), the set $\set{r>L:\psi(0,r)-Wr^2/2>\gmm}$ is given by an interval $(L,R)$ for some $R\geq R_2$. Hence, from the Steiner symmetry of $\psi$, we deduce that the superlevel set $\set{\psi-Wr^2/2>\gmm}\cap\set{r>L}$ is simply connected and has the Steiner symmetry. It completes the proof of \textit{(Claim 5)}.\\

\noindent$\blacktriangleright$ \textit{(Claim 6) The superlevel set $\set{\psi-Wr^2/2>\gmm}\cap\set{r>L}$ is given by
$$\set{(z,r)\in\Pi:r\in(L,R),\,|z|<l(r)}$$ for some function $l:[L,R]\to[0,\ift)$ satisfying $l>0$ in $(L,R)$, $l(L)=l(R)=0$, and $l\in C([L,R])\cap C^1((L,R))$.}

\medskip
The existence of such a curve $l:(L,R)\to[0,\ift)$ is derived from \textit{(Claim 5)} such that $l>0$ in $(L,R)$. At the same time, the differentiability of $l$ and the limits $$l(L):=\lim_{s\downarrow L}l(s)=0,\qd l(R):=\lim_{s\uparrow R}l(s)=0$$ are obvious from the \textit{strict} Steiner symmetry and differentiability of $\psi$. It completes the proof of \textit{(Claim 6)}.\\

\noindent$\blacktriangleright$ \textit{(Claim 7) The corresponding vortex domain is contained in $\set{r>L}$.}

\medskip
Let $(z',r')\in\Pi$ be any point satisfying $r'\leq L$. From the definition of $L>0$, we deduce that the function $s\mapsto\psi(0,s)-Ws^2/2$ is negative and decreases strictly in the interval $(0,L]$. Hence, we get 
\begin{equation}\label{eq1118_1}
\psi(0,r')-\f{W}{2}r'^2<\psi(0,r)-\f{W}{2}r^2<0\qd\mbox{for each}\q r\in(0,r').
\end{equation}

\begin{itemize}
\item[$\ast$] Assume \say{$z'\leq0$}. We suppose the contrary, i.e., $(z',r')$ is in the vortex domain. Since the vortex domain is open in $\Pi$, there exists a small constant $0<\eps\ll1$ such that the point $(z'-\eps, r'-\eps)$ is still in the vortex domain, and the relation \eqref{eq1118_1} gives us
$$\psi(z'-\eps,r'-\eps)-\f{W}{2}(r'-\eps)^2<\psi(0,r'-\eps)-\f{W}{2}(r'-\eps)^2<\psi(0,r)-\f{W}{2}r^2<0\qd\mbox{for each}\q r\in(0,r'-\eps).$$ Applying Lemma~\ref{lem: 3D streamline formation} to the above relation guarantees that $(z'-\eps,r'-\eps)$ cannot be contained in the vortex domain, so a contradiction (the reason we have considered $(z'-\eps,r'-\eps)$ instead of $(z',r')$ is that we wanted to avoid the case $(z',r')=(0,L)$).\medskip
\item[$\ast$] Assume \say{$z'>0$}. We suppose the contrary, i.e., $(z',r')$ is in the vortex domain. Here, we have two possibilities: 
$$\text{(I) }\psi(z',r')-Wr'^2/2\geq\gmm\qd\mbox{or}\qd
\text{(II) }\psi(z',r')-Wr'^2/2<\gmm.$$
\begin{itemize}
\item[(I):] We first consider the case (I). Since the vortex domain is open, for small $0<\eps\ll z'$ the point $(z'-\eps,r')$ is also in the vortex domain while $z'-\eps>0$ and 
$$\gmm<\sgm_1:=\psi(z'-\eps,r')-\f{W}{2}r'^2<\psi(0,r')-\f{W}{2}r'^2$$
by the \textit{strict} Steiner symmetry of $\psi$. Note that it implies $r'<L$ (by the definition of $\gmm:=\psi(0,L)-WL^2/2$). For each $r\in[r',L)$, we now consider the equation $$\psi(z,r)-\f{W}{2}r^2=\sgm_1\qd\mbox{for}\q z\geq0.$$ Note that the map $s\mapsto \psi(0,s)-Ws^2/2$ is decreasing in the interval $(0,L)$ (by \textit{(Claim 1)}) and attains the value greater than $\sgm_1$ at $s=r'$ until it reaches $\gmm$ at $s=L$, the above equation is solved as $z=f(r)$ for some $f:[r',L']$ with some constant $L'\in(r',L)$ satisfying $f>0$ in $[r',L')$ with $f(L'):=\lim_{s\uparrow L'}f(s)=0$ (the differentiability of $f$ is obvious from the differentiability and the \textit{strict} Steiner symmetry of $\psi$). Then the particle starting at $(z'-\eps,r')$ increases its $r$-coordinate along the curve $(f(r),r)$ to reach $(0,L')$ (by \textbf{[P3]}) and crosses the symmetry axis ($r$-axis in the moving frame) of the vortex ring $\overline\xi$ in a finite time (recall \textit{(Claim 1)}; the horizontal speed $\overline{v}^z(0,L')=\rd_r\psi(0,L')/L'$ is strictly less than the speed $W$ of the vortex ring $\overline\xi$), which brings us back to the case \say{$z'\leq0$} that has been already negated (for a more rigorous explanation, we refer to \textbf{Case \circled{2}} in the proof of Theorem~\ref{thm: class. of 2D atmos.}, which is containing essentially same arguments with the current setting).
\medskip
\item[(II):] Now we consider the case (II): $\psi(z',r')-Wr'^2/2=:\sgm_2<\gmm$. Now, for each $r\geq r'$, we consider the equation 
\begin{equation}\label{eq1118_2}
\psi(z,r)-\f{W}{2}r^2=\sgm_2\qd\mbox{for}\q z\geq0.
\end{equation}
Since we know 
\begin{equation}\label{eq1118_3}
\psi(0,r)-Wr^2\geq\gmm>\sgm_2\qd\mbox{for each}\q r\in(0,R_2)
\end{equation}
(by \textit{(Claim 2)} and \textit{(Claim 3)}) and  
the map $s\mapsto \psi(0,s)-Ws^2/2$ decreases strictly in the interval $[R_2,\ift)$ to achieve the value $\sgm_2$ for some $s\geq R_2$ (by \textbf{[P2]}), the equation \eqref{eq1118_2} above is solved as $z=\tld{f}(r)$ for some function $\tld{f}:[r',L'']\to[0,\ift)$ with some constant $L''>R_2$ satisfying $\tld{f}>0$ in $[r',L'')$ with $\tld{f}(L''):=\lim_{s\uparrow L''}\tld{f}(s)=0$ (the differentiability of $\tld{f}$ is obvious from the differentiability and the \textit{strict} Steiner symmetry of $\psi$). Similarly with the case (I), the particle starting at $(z',r')$ increases its $r$-coordinate along the curve $(\tld{f}(r),r)$ to reach $(0,L'')$ and crosses the $r$-axis in a finite time (since $\overline{v}^z(0,L'')=\rd_r\psi(0,L'')/L''<0$ by \textbf{[P2]}), which brings us back to the case \say{$z'\leq0$} that has been already negated.
\end{itemize} 
\end{itemize}\medskip
In sum, we complete the proof of \textit{(Claim 7)}.\\

\noindent$\blacktriangleright$ \textit{(Claim 8) The superlevel set $\set{\psi-Wr^2/2>\gmm}\cap\set{r>L}$ is the corresponding vortex domain.}

\medskip
We will first show that the superlevel set $\set{\psi-Wr^2/2>\gmm}\cap\set{r>L}$ is contained in the vortex domain. \textit{(Claim 5)} says that this superlevel set is bounded. For any point $(z,r)$ in this superlevel set, under the 3D axisymmetric flow map $\Phi_3=(\Phi_3^z,\Phi_3^r)$ generated by 
$\xi(t,\cdot):=\overline\xi(\cdot-Wt\bfe_z)$, we have 
$$\psi(\Phi_3(t,(z,r))-Wt\bfe_z)-\f{W}{2}\Phi_3^r(t,(z,r))^2
=\psi(z,r)-\f{W}{2}r^2>\gmm\qd\mbox{for each}\q t\geq0.$$Knowing that $$\psi(s,L)-WL^2/2\leq \psi(0,L)-WL^2/2=\gmm\qd\mbox{for any}\q s\in\bbR,$$ the particle trajectory starting at $(z,r)$ never crosses the line $\set{r=L}$. Instead, it stays in the superlevel set for all time $t\geq0.$ In sum, the superlevel set is contained in the vortex domain. 

Now, we will show that the superlevel set $\set{\psi-Wr^2/2>\gmm}\cap\set{r>L}$ is indeed the vortex domain. We already proved in \textit{(Claim 7)} that the vortex domain is contained in $\set{r>L}$. Knowing this, we suppose the contrary, i.e., there exists a point $(z',r')$ in the vortex domain satisfying $r'>L$ and $\psi(z',r')-Wr'^2/2\leq\gmm$. Since the vortex domain is an open set in $\Pi$, there exists small constant $0<|\eps|\ll1$ such that the point $(z'+\eps,r')$ is in the vortex domain while $z'+\eps\neq0$ and $\psi(z'+\eps,r')-Wr'^2/2=:\gmm'<\gmm$. Meanwhile, we observe that
\begin{equation}\label{eq1106_4}
\gmm'<\psi(0,r)-\f{W}{2}r^2\qd\mbox{for each}\q r\in(0,r'].
\end{equation}
Indeed, above relation is obvious for the case $r'\in(0,R]$ by \textit{(Claim 1)} and \textit{(Claim 5)}. For the case $r'>R$, we note that $s\mapsto\psi(0,s)-Ws^2/2$  decreases strictly in the interval $[R,\ift)\subset[R_2,\ift)$ (by \textbf{[P2]}), implying $$\psi(0,r)-\f{W}{2}r^2>\psi(0,r')-\f{W}{2}r'^2>\psi(z'+\eps,r')-\f{W}{2}r'^2=\gmm'\qd\mbox{for each}\q r\in[R,r')$$ and hence confirming \eqref{eq1106_4}.

\medskip
\noindent\textbf{Case \circled{1}: $z'+\eps<0$}

Applying Lemma~\ref{lem: 3D streamline formation} to the relation \eqref{eq1106_4}, we derive that the point $(z'+\eps,r')$ cannot be contained in the vortex domain, which is a contradiction. So \textbf{Case \circled{1}} is impossible.

\medskip
\noindent\textbf{Case \circled{2}: $z'+\eps>0$} 

The idea is basically the same as the case (II) with \say{$z'>0$} in the proof of \textit{(Claim 7)}, and we will give a proof for completeness. We will derive a contradiction by finding the trajectory starting from $(z'+\eps,r')$ that intrudes the region $\set{z<0}$ in a finite time and recalling that it has been negated in \textbf{Case \circled{1}} above. By the relation \eqref{eq1106_4}, we can solve the equation
$$\psi(z,r)-\f{W}{2}r^2=\gmm'\qd\mbox{for}\q z\geq0,$$ for each fixed $r>0$. This equation is solvable for each $r\in[r', R']$ for some $R'>R$ satisfying $\psi(0,R')-WR'^2/2=\gmm'$ (the relation $R'>R$ is deduced from the fact that $\psi(0,s)-Ws^2/2$ decreases strictly in $s\geq R_2$ by \textbf{[P2]}), and the solution is given as $z=g(r)$ for some function $g:[r',R']\to[0,\ift)$ satisfying $g>0$ in $(r',R')$, $g(r')=z'+\eps$, and $g(R'):=\lim_{s\uparrow R'}g(s)=0$ (the differentiability of $g$ is obvious from the differentiability and the \textit{strict} Steiner symmetry of $\psi$). Then the particle starting at $(z'+\eps,r')$ increases its $r$-coordinate along the curve $(g(r),r)$ to reach $(0,R')$ (by \textbf{[P3]}) and crosses the $r$-axis in a finite time (since $\overline{v}^z(0,R')=\rd_r\psi(0,R')/R'<0$ by \textbf{[P2]}), which brings us back to \textbf{Case \circled{1}} that has already been negated. Hence, \textbf{Case \circled{2}} is also impossible.\\

It completes the proof of [Case II] in Theorem~\ref{thm: class. of 3D atmos.}. \end{proof}

\section{Open problems}\label{sec: open problems}

In this final section, we give some open questions from a mathematical viewpoint.

\medskip
\begin{enumerate}
\item \textbf{Vortex atmosphere for general settings.}

The notion of a vortex atmosphere naturally extends far beyond the setting of steadily translating vortices; for instance, it can be considered for co-rotating vortex pairs and even for non-rigid vortex motions. In 1992, Meleshko et al.~\cite{MKGK1992} investigated the dynamics of vortex atmospheres in various multi-vortex configurations, such as the interaction of two vortex pairs with different circulations (see also \cite{MGK2011}). In particular, they examined vortices exhibiting leapfrogging behavior and observed the existence and evolution of a \textit{leapfrogging atmosphere}.

\medskip

Given the recent surge of interest in the mathematical analysis of leapfrogging vortices \cite{HHM2025, BG2019, GSPBB2021, Aiki2019, Acheson2000, TA2013, DPMW2024, DHLM2025, BCM2025, CFQW2025, BKM2013, SM2003}, determining whether a vortex atmosphere exists in a leapfrogging scenario--and, if so, understanding its geometric structure--appears to be a particularly compelling open problem. 

\medskip
\item \textbf{Stability of vortex atmosphere.}

Research on the stability of vortex solutions with steady (or rigid) motion is extensive (see, e.g., \cite{Wan1986, Wan88, BNL2013, BCM2019, AC2022, CL22, IJ2022, CJY2024, GS2024, AJY2025, Wang2024, Choi2024, HHR2024, CJS2025, ACJ2025, CLW2025, CQZZ2025}), and most studies have focused on the stability of the vorticity (or its stream function). See also \cite{BM2015, IJ2023, MZ23} for shear flow settings.
\medskip

However, the focus of stability analysis can instead be shifted from the vorticity to the vortex atmosphere. For instance, one may investigate the stability of a given vortex atmosphere by examining the topological or quantitative changes in its streamlines under perturbations. Such an analysis requires a notion of the vortex atmosphere for vortex solutions exhibiting non-rigid motion, as discussed in (1) above.

\medskip
\item \textbf{Ratio of $|Atmosphere|:|Core|$.}

Hicks \cite{Hicks1919} calculated the volume of the atmosphere (the liquid carried forward) across a range of ratios between the core radius and ring radius. He pointed out that the thinner the ring is (compared to the ring's size), the thinner its toroidal-shape atmosphere becomes. However, to our knowledge, the asymptotic rate of $|Atmosphere|:|Core|$  remains open.

\medskip
\item \textbf{Existence of revolved-lemniscate atmosphere.}

Although the atmosphere of revolved lemniscate type (Figure~\ref{subfig: lemniscate} and the 
case $\overline{v}^z(0,0)=W$ in \textbf{[Case I]} of Theorem~\ref{thm: class. of 3D atmos.}) 
has been discussed in the literature (e.g., \cite{Hicks1919, Eisenga1997, MS2022}), its rigorous existence remains open. Unlike the spheroidal and toroidal cases 
(Theorems~\ref{thm: 3D sphere} and \ref{thm: 3D donut}), for which Hill’s spherical vortex 
and the line-vortex ring serve as benchmark solutions, establishing the existence of 
a revolved lemniscate atmosphere appears considerably more challenging, owing to the absence 
of a suitable model problem that could serve as an analogue.

\end{enumerate}

\appendix
\section{Proof of Proposition~\ref{prop: decay of stream/term}}\label{appendix}
\begin{proof}[Proof of Proposition~\ref{prop: decay of stream/term}]
For a given axisymmetric $\overline\xi\in L^\ift(\bbR^3)$ with a compact support, its 3D axisymmetric stream function $\psi=\psi[\overline\xi]$ satisfies $|\psi/r|=O(|(z,r)|^{-2})$ as 
$|(z,r)|\to\ift$ (see \cite[Lemma 1.1]{FT1981}), which gives 
$|\psi(z,r)/r^2|\leq C_0 r^{-1}|(z,r)|^{-2}\leq C_0 r^{-3}$ as $|(z,r)|\gg1$, for some constant $C_0>0$. Moreover, as $\psi(z,0)=0$ for each $z\in\bbR$, the mean value theorem tells us that for each $z\in\bbR$ and $r>0$, $$\f{\psi(z,r)}{r^2}=\f{(\rd_r\psi)(z,r')}{r}\qd\mbox{for some}\q r'=r'(z,r)\in(0,r).$$ For each $(z,r)$ satisfying $|z|\gg1$, we get $|(z,r')|\gg1$,
and hence we apply the relation $|\nb\psi/r|=O(|(z,r)|^{-3})$ as 
$|(z,r)|\to\ift$ (see \cite[Lemma 1.1]{FT1981}) to obtain that
$|\psi(z,r)/r^2|\leq |\nb\psi(z,r')/r'|\leq C_1 |(z,r')|^{-3}\leq C_1 |z|^{-3}$ for some constant $C_1>0$. Now, we consider a constant $R\gg1$ large enough and any point $(z,r)$ satisfying $|(z,r)|\geq R$. Then we have either $|z|\geq R/\sqrt2\gg1$ or $r\geq R/\sqrt2\gg1$. For the former case, we use the relation $|\psi(z,r)/r^2|\leq C_1 |z|^{-3}$, while for the later case, we apply $|\psi(z,r)/r^2|\leq C_0 r^{-3}$. In sum, we get $$\left|\f{\psi(z,r)}{r^2}\right|\leq C_2 R^{-3}\qd\mbox{for}\q R\gg1$$ for the constant $C_2=2^{3/2}\max\set{C_0,C_1}$, as desired.

\medskip
Similarly, for a given odd symmetric $\overline\omg\in L^\ift(\bbR^2)$ and its stream function $\calG=\calG[\overline\omg]$, we borrow the estimate $|\calG(x_1,x_2)|\leq C_3 |x_2|^{1/2}$ for some constant $C_3>0$ from \cite[Proposition 2.1]{AC2022}, and we obtain that $|\calG(x_1,x_2)/x_2|\leq C_3 |x_2|^{-1/2}$ if $x_2\neq0$. On the other hand, 
since $\calG(x_1,0)=0$ for each $x_1\in\bbR$, the mean value theorem once again gives that for each $(x_1,x_2)\in\bbR$ with $x_2\neq0$,
$$\f{\calG(x_1,x_2)}{x_2}=\left(\rd_{x_2}\calG\right)(x_1,x_2')\qd\mbox{for some}\q x_2'=x_2'(x_1,x_2)\in(\min\set{0,x_2},\max\set{0,x_2}).$$ Applying the boundedness of the support of $\overline\omg$ to the relation $|\nb\calG(\bfx)|\leq C_4\int_{\bbR^2}\f{1}{|\bfx-\bfy|}|\overline\omg(\bfy)|\dd\bfy$  for some constant $C_4>0$, we have $|\nb\calG(\bfx)|=O(|\bfx|^{-1})$ as $|\bfx|\gg1$. Hence, for any $(x_1,x_2)\in\bbR^2$ such that $|x_1|\gg1$, we have $|\calG(x_1,x_2)/x_2|\leq |\nb\calG(x_1,x_2')|\leq C_5|(x_1,x_2')|^{-1}\leq C_5|x_1|^{-1}$ for some constant $C_5>0$. As with the case of the axisymmetric case $\overline\xi$, we consider a large enough $R\gg1$, and suppose a point $(x_1,x_2)\in\bbR^2$ satisfies $|(x_1,x_2)|\geq R$. Then we have either $|x_1|\geq R/\sqrt2\gg1$ or $|x_2|\geq R/\sqrt2\gg1$, which leads to $|\calG(x_1,x_2)/x_2|\leq 2^{1/2} C_5 R^{-1}$ or $|\calG(x_1,x_2)/x_2|\leq 2^{1/4}C_3 R^{-1/2}$ respectively, as desired. It completes the proof.

\end{proof}

\q \section*{Acknowledgments}

KC and YS were supported by the National Research Foundation of Korea (NRF) grant funded by the Korean government (MSIT) (No. RS-2023-00274499). IJ was supported by the NRF grant from the Korea government (MSIT), No. 2022R1C1C1011051, RS-2024-00406821.

\bibliographystyle{plain}
%\bibliography{biblography(Sim)}

\end{document}